\documentclass[UKenglish, letterpaper]{amsart}

\usepackage{amscd}
\usepackage{amssymb}
\usepackage[T1]{fontenc}
\usepackage[utf8]{inputenc}
\usepackage{hyperref}
\usepackage{varioref}
\usepackage[hyperpageref]{backref}
\usepackage[arrow,curve,matrix]{xy}

\usepackage{mathpazo}
\usepackage{mathrsfs}

\usepackage{graphicx}
\usepackage{color}
\usepackage{tikz}
\usetikzlibrary{arrows,shapes,trees} 

\def\clap#1{\hbox to 0pt{\hss#1\hss}}

\def\mathclap{\mathpalette\mathclapinternal}

\def\mathclapinternal#1#2{%
\clap{$\mathsurround=0pt#1{#2}$}}

\DeclareFontFamily{OMS}{rsfs}{\skewchar\font'60}
\DeclareFontShape{OMS}{rsfs}{m}{n}{<-5>rsfs5 <5-7>rsfs7 <7->rsfs10 }{}
\DeclareSymbolFont{rsfs}{OMS}{rsfs}{m}{n}
\DeclareSymbolFontAlphabet{\scr}{rsfs}

\newcommand{\sA}{\scr{A}}
\newcommand{\sB}{\scr{B}}
\newcommand{\sC}{\scr{C}}

\newcommand{\sF}{\scr{F}}
\newcommand{\sG}{\scr{G}}
\newcommand{\sH}{\scr{H}}

\newcommand{\sJ}{\scr{J}}

\newcommand{\sL}{\scr{L}}

\newcommand{\sO}{\scr{O}}

\newcommand{\sQ}{\scr{Q}}

\newcommand{\sS}{\scr{S}}
\newcommand{\sT}{\scr{T}}

\newcommand{\bC}{\mathbb{C}}

\newcommand{\bN}{\mathbb{N}}

\newcommand{\bP}{\mathbb{P}}
\newcommand{\bQ}{\mathbb{Q}}
\newcommand{\bR}{\mathbb{R}}

\newcommand{\bZ}{\mathbb{Z}}

\DeclareMathOperator{\codim}{codim}

\DeclareMathOperator{\Hom}{Hom}

\DeclareMathOperator{\NE}{NE}

\DeclareMathOperator{\Pic}{Pic}
\DeclareMathOperator{\rank}{rank}

\DeclareMathOperator{\reg}{reg}

\DeclareMathOperator{\Sym}{Sym}
\DeclareMathOperator{\supp}{supp}
\DeclareMathOperator{\tor}{tor}

\DeclareMathOperator{\nklt}{{nklt}}

\newcommand{\wtilde}{\widetilde}
\newcommand{\what}{\widehat}

\newcommand{\sHom}{\scr{H}\negmedspace om}

\newcounter{thisthm}

\newcommand{\iref}[1]{(\thesection.\the\value{thisthm}.\the\value{#1})}

\theoremstyle{plain}    
\newtheorem{thm}{Theorem}[section]
\newtheorem{defn}[thm]{Definition}

\newtheorem{assumption}[thm]{Assumption}

\numberwithin{equation}{thm}
\numberwithin{figure}{section}
\theoremstyle{plain}    
\newtheorem{cor}[thm]{Corollary}
\newtheorem{lem}[thm]{Lemma}
\newtheorem{awlog}[thm]{Assumption without loss of generality}

\newtheorem{consequence}[thm]{Consequence}

\newtheorem{fact}[thm]{Fact}
\newtheorem{question}[thm]{Questions}

\theoremstyle{plain}    
\newtheorem{prop}[thm]{Proposition}
\newtheorem{proclaim-special}[thm]{\specialthmname}

\theoremstyle{remark}
\newtheorem{rem}[thm]{Remark}
\newtheorem{remq}[thm]{Remark and Question}
\newtheorem{obs}[thm]{Observation}

\newtheorem{subrem}[equation]{Remark}
\newtheorem*{claim*}{Claim} 

\newtheorem{construction}[thm]{Construction}
 
\newtheorem{example}[thm]{Example}
\newtheorem{setting}[thm]{Setting}

\newtheoremstyle{bozont-remark}{3pt}{3pt}%
     {}
     {}
     {\it}
     {.}
     {.5em}
     {\thmname{#1}\thmnumber{ #2}: \thmnote{\sc #3}}
\theoremstyle{bozont-remark}

\def\factor#1.#2.{\left. \raise 2pt\hbox{$#1$} \right/\hskip -2pt\raise
  -2pt\hbox{$#2$}} 
\newlength{\swidth}
\newenvironment{enumerate-p}{
  \begin{enumerate}}
  {\setcounter{equation}{\value{enumi}}\end{enumerate}}

\definecolor{tomato}{RGB}{180,62,39}
\definecolor{forrest}{RGB}{81,133,49}
\definecolor{lighttomato}{RGB}{253,65,65}
\definecolor{lightforrest}{RGB}{145,237,87}
\definecolor{mygreen}{RGB}{40,104,69}
\definecolor{mygreen2}{RGB}{3,149,39}
\definecolor{darkolivegreen}{RGB}{102,118,75}
\definecolor{cranegreen}{RGB}{102,118,75}
\definecolor{mydarkblue}{RGB}{10,92,153}
\definecolor{myblue}{RGB}{57,222,186}
\definecolor{pinkish}{RGB}{213,83,222}
\definecolor{colD}{RGB}{213,83,222}
\definecolor{defb}{RGB}{213,83,222}
\definecolor{goldenrod}{RGB}{225,115,69}
\definecolor{mauve}{RGB}{224, 176, 255}
\definecolor{fuchsia}{RGB}{255, 0, 255}
\definecolor{lavender}{RGB}{230, 230, 250}
\definecolor{gold}{RGB}{255, 215, 0}
\definecolor{orange}{RGB}{255, 127, 0}
\definecolor{maroon}{RGB}{123, 17, 19}
\definecolor{brightmaroon}{RGB}{195, 33, 72}
\definecolor{richmaroon}{RGB}{176, 48, 96}
\definecolor{green}{RGB}{3,149,39}

\title[Reflexive differential forms on singular spaces]{Reflexive differential forms on singular spaces -- Geometry and Cohomology}

\author{Daniel Greb}
\author{Stefan Kebekus}
\author{Thomas Peternell}

\thanks{All three authors were supported in part by the DFG-Forschergruppe
  ``Classification of Algebraic Surfaces and Compact Complex Manifolds''. Parts
  of this paper were written while the first named author enjoyed the hospitality
  of the Mathematics Department at Princeton University. He gratefully
  acknowledges the support of the Baden--Württemberg--Stiftung through the
  ``Eliteprogramm für Postdoktorandinnen und Postdoktoranden''.}


\address{Daniel Greb, Mathematisches Institut, Albert-Ludwigs-Universität
  Freiburg, Eckerstra{\ss}e 1, 79104 Freiburg im Breisgau, Germany}
\email{\href{mailto:daniel.greb@math.uni-freiburg.de}{daniel.greb@math.uni-freiburg.de}}
\urladdr{\href{http://home.mathematik.uni-freiburg.de/dgreb}{http://home.mathematik.uni-freiburg.de/dgreb}}

\address{Stefan Kebekus, Mathematisches Institut, Albert-Ludwigs-Universität
  Freiburg, Eckerstra{\ss}e 1, 79104 Freiburg im Breisgau, Germany}
\email{\href{mailto:stefan.kebekus@math.uni-freiburg.de}{stefan.kebekus@math.uni-freiburg.de}}
\urladdr{\href{http://home.mathematik.uni-freiburg.de/kebekus}{http://home.mathematik.uni-freiburg.de/kebekus}}

\address{Thomas Peternell, Mathematisches Institut, Universität
  Bayreuth, 95440~Bayreuth, Germany}
\email{\href{mailto:thomas.peternell@uni-bayreuth.de}{thomas.peternell@uni-bayreuth.de}}
\urladdr{\href{http://btm8x5.mat.uni-bayreuth.de/mathe1}{http://btm8x5.mat.uni-bayreuth.de/mathe1}}

\hypersetup{
  pdfauthor={Daniel Greb, Stefan Kebekus, and Thomas Peternell},
  pdftitle={Reflexive differential forms on singular spaces -- Geometry and Cohomology},
  pdfsubject={Primary: 14F10; Secondary: 14F17, 14E30, 32S05, 32S20},
  pdfkeywords={singularities of the minimal model program, differential forms, vanishing theorems, Hodge theory},
  pdfstartview={Fit},
  pdfpagelayout={TwoColumnRight},
  pdfpagemode={UseOutlines},
  bookmarks,
  colorlinks}

\date{\today}

\begin{document}

\begin{abstract}
  Based on a recent extension theorem for reflexive differential forms, that is,
  regular differential forms defined on the smooth locus of a possibly singular
  variety, we study the geometry and cohomology of sheaves of reflexive
  differentials.
  
  First, we generalise the extension theorem to holomorphic forms on locally
  algebraic complex spaces.  We investigate the (non-)existence of reflexive
  pluri-differentials on singular rationally connected varieties, using a
  semistability analysis with respect to movable curve classes. The necessary
  foundational material concerning this stability notion is developed in an
  appendix to the paper.  Moreover, we prove that Kodaira--Akizuki--Nakano
  vanishing for sheaves of reflexive differentials holds in certain extreme
  cases, and that it fails in general. Finally, topological and Hodge-theoretic
  properties of reflexive differentials are explored.
\end{abstract}

\maketitle

\tableofcontents

\section{Introduction}\label{sect:intro}


Holomorphic differential forms are an indispensable tool to study the global
geometry of non-singular projective varieties and compact Kähler manifolds. In
the presence of singularities, with the exception of forms of degree one and
forms of top degree, the influence of differential forms on the geometry of a
variety is much less explored. At the same time, in higher-dimensional algebraic
geometry one is naturally led to consider singular varieties in a given
birational equivalence class, even if one is primarily interested in projective
manifolds. Hence, the need for a comprehensive geometric and cohomological
theory of differential forms on singular spaces arises.

In this paper we study the geometry and cohomology of sheaves of \emph{reflexive
  differential forms}, that is, regular forms defined on the smooth locus of a
possibly singular variety. Motivated by and using the techniques of the Minimal
Model Program we focus on the natural classes of singularities in Mori Theory:
log canonical, kawamata log terminal, canonical, and terminal. For definitions
and a thorough discussion of these classes of singularities we refer the reader
to \cite[Chap.~2.3]{KM98} and \cite{KollarSingsOfPairs}.

The basic tool for our investigation is an extension theorem for reflexive
differentials, which we state here in a simplified version.

\begin{thm}[Extension Theorem]\label{thm:ext}
  Let $X$ be a quasi-projective variety and $D$ an effective $\mathbb Q$-divisor such that
  the pair $(X,D) $ is Kawamata log terminal (klt). Let $\pi: \wtilde X \to X$ a
  be log resolution of the pair $(X,D)$. Then, the sheaves
  $\pi_*(\Omega^p_{\wtilde X})$ are reflexive for $1 \leq p \leq n$.
\end{thm}

In fact, a much more general statement holds, for which we refer the reader to
\cite[Thm.~1.5]{GKKP11} or to the survey \cite{Keb11}.

Introducing the sheaves $\Omega^{[p]}_X = (\bigwedge^p \Omega^1_X)^{**}$ of
reflexive differentials, Theorem~\ref{thm:ext} can be rephrased by saying that
$$
\pi_*(\Omega^p_{\tilde X}) = \Omega^{[p]}_X\;\;\;\; \forall p \in \{1, \dots,
n\},
$$
or in yet another way by saying that any regular differential form defined on
the smooth part of a pair $(X, D)$ with at worst klt singularities defines a
regular differential form on any log resolution of $(X, D)$.

We will use this result (as well as a generalisation to holomorphic differential
forms on locally algebraic varieties established here) to study a number of
problems about existence of differential forms, vanishing theorems, and
Hodge-theoretic properties of reflexive differentials.

\subsection*{Outline of the paper}

We conclude this introduction by a summary of the contents of the paper.

\subsubsection*{Extension theorems in the holomorphic setting}

The Extension Theorem \ref{thm:ext} is established in the algebraic
category. Although some of the methods employed in \cite{GKKP11} to prove
Theorem~\ref{thm:ext}, especially the use of the Minimal Model Program and
certain vanishing theorems, are not at one's disposal in the analytic context,
one should nevertheless expect the extension property to hold also in the
holomorphic setting. In Section~\ref{sect:holext} we prove this for klt complex
spaces $X$ that locally algebraic:

\begin{thm}[Extension Theorem for holomorphic forms]\label{thm:anaext} 
  Let $X$ be a normal complex space and $D$ an effective $\mathbb Q$-divisor such that the
  pair $(X,D) $ is analytically klt. Suppose $(X,D)$ is locally algebraic.  Let
  $\pi: \wtilde X \to X$ a be log resolution of the pair $(X,D)$. Then, the
  sheaves $\pi_*(\Omega^p_{\wtilde X})$ are reflexive for $1 \leq p \leq n$.
\end{thm}

In fact, analogous to \cite[Thm.~1.5]{GKKP11} we prove a much more general
result for locally algebraic log canonical pairs, see
Theorem~\ref{thm:holomorphicextension}. Using Artin's Algebraization Theorem, an
extension theorem for normal complex spaces with isolated singularities is
derived as an immediate consequence of Theorem~\ref{thm:anaext}, see
Corollary~\ref{cor:isolated}.

\subsubsection*{Reflexive differentials on rationally chain connected spaces}

Rationally connected and, more generally, rationally chain connected projective
varieties play a key role in the classification and structure theory of
algebraic varieties. It is a fundamental fact of higher-dimensional algebraic
geometry that a rationally connected projective manifold $X$ does not carry any
non-trivial differential forms. This statement also holds for reflexive
differential forms on rationally connected varieties with at worst klt
singularities, see~\cite[Thm.~5.1]{GKKP11}. In the smooth case, considering
general tensor powers instead of exterior powers of $\Omega_X^1$, it can even be
shown that a rationally connected projective manifold does not carry any
\emph{pluri-form}:
\begin{equation}\label{eq:Mumford}
  H^0\bigl(X, \,(\Omega^1_X)^{\otimes m}\bigr)
  = 0 \;\;\;\; \forall m \in \mathbb{N}^{\geq 1}.
\end{equation}
Exploring the singular setup, in Section~\ref{sect:rcc} we derive an analogous
vanishing result for reflexive pluri-forms on \emph{factorial} klt spaces:

\begin{thm}[cf.~Theorem~\ref{thm:rc}]\label{thm:factorialMumford}
  Let $X$ be a rationally chain-connected factorial klt space. Then
  $$
  H^0\bigl(X, \, ((\Omega^1_X)^{\otimes m})^{**} \bigr) = 0.
  $$
\end{thm}

Note here that factorial klt spaces automatically have canonical
singularities. Moreover, they are rationally connected by a result of Hacon and
McKernan, cf.~Remark~\ref{rem:rcc}. The above generalisation of
\eqref{eq:Mumford} to the singular case is not at all obvious and, as a matter
of fact, surprisingly fails in the non-factorial canonical case. A corresponding
two-dimensional example is provided in Section~\ref{ssec:C1}. The proof of
Theorem~\ref{thm:factorialMumford} relies on a semistability analysis of sheaves
of reflexive differentials with respect to movable curves on klt spaces. The
necessary foundational material concerning this semistability notion is
developed in detail in Appendix~\ref{sec:snsp}; see below for an introduction.

\subsubsection*{Vanishing theorems}

Section~\ref{sect:KAN} of the paper discusses vanishing theorems for reflexive
differentials on klt spaces. In the classical smooth setting, the
Kodaira-Akizuki-Nakano theorem states that for any ample line bundle $\sL$ on a
projective manifold $X$ we have
\begin{align}
  H^q \bigl(X,\,\Omega^{p}_X \otimes \sL \bigr) & = 0 \quad \text{for $p + q
    > n$, and} \label{eq:1}\\
  H^q \bigl(X,\,\Omega^{p}_X \otimes \sL^{-1} \bigr) & = 0 \quad \text{for
    $p + q < n$.}\label {eq:2}
\end{align}
We will prove in Section~\ref{ssec:pKANvan} that these vanishings still hold on
klt spaces in certain extremal ranges, again with $\Omega^p_X$ replaced by
$\Omega_X^{[p]}$. At the same time we show that the generalisation of
\eqref{eq:1} and \eqref{eq:2} to sheaves of reflexive differentials on klt
spaces fails in general. In fact, in Section~\ref{ssec:nopKANvan} we exhibit a
projective $4$-fold $X$ with a single isolated terminal Gorenstein singularity
carrying an ample line bundle $\sL$ such that
$$
H^2\bigl(X,\,\Omega^{[3]}_X \otimes \sL\bigr) \simeq H^2\bigl(X,\,\Omega^{[1]}_X
\otimes \sL^{-1}\bigr) \simeq \mathbb C,
$$
We do not know at the moment whether a similar counterexample also exists in
dimension $3$.

\subsubsection*{Hodge-Theory and the Poincar\'e-Lemma}
It follows from standard Hodge theory for compact Kähler manifolds that every
regular differential form on a projective manifold is closed. While the
existence of a suitable Hodge theory for reflexive differentials is still an
open question, the Extension Theorem~\cite[Thm.~1.5]{GKKP11} implies the
following generalisation of this particular Hodge-theoretic result to log
canonical pairs, which we show in Section~\ref{sect:Poincare}.

\begin{prop}[Closedness of global logarithmic forms]
  Let $X$ be a projective variety, and let $D$ be a $\mathbb{Q}$-divisor such
  that the pair $(X, D )$ is log-canonical. Then, any reflexive logarithmic
  $p$-form $\sigma \in H^0\bigl(X,\, \Omega^{[p]}_X(\log \lfloor D \rfloor)
  \bigr)$ is closed.
\end{prop}

Next, it is shown that reflexive one-forms on klt spaces satisfy a Poincaré
lemma. As a corollary, we obtain that every globally defined closed reflexive one-form
is represented in a canonical manner by a Kähler differential.

\subsubsection*{Semistability notions on singular spaces}

In Appendix~\ref{sec:snsp} we develop a semistability theory for torsion-free
sheaves with respect to a movable class on a klt space, generalizing the results
presented in \cite{CP11}.  Specifically, we establish the existence of a maximal
destabilizing subsheaf and prove that semistability is preserved under tensor
operations.

\subsubsection*{General Remark}
Throughout the paper we will use fundamental facts about reflexive differentials
on singular varieties, as well as basic definitions concerning logarithmic pairs
and resolution of singularities, for which we refer the reader to
\cite[Sect.~2]{GKKP11} and \cite[Sect.~2]{GKK08}.

\subsection*{Acknowledgements }

Daniel Greb wants to thank János Kollár for interesting discussions concerning
the contents of this paper. The authors also want to thank Sándor Kovács and
Luca Migliorini for numerous discussions and hints. They also thank the referee
for reading the paper extremely carefully, and for several helpful suggestions
which led to a partial generalization of Proposition~\ref{prop:partVan2}, and to
a substantial simplification of the argument in Appendix~\ref{ssec:swrtm}.

\section{Extension of holomorphic differential forms}
\label{sect:holext}

In this section we discuss the extension properties of holomorphic differential
forms on spaces with at worst log canonical singularities. In order to
distinguish the euclidean topology of a given complex space from its Zariski
topology, an open subset in the earlier topology will be called ``open'', while
open subsets in the latter topology will be called ``Zariski-open''. A
corresponding convention will be used for closed subsets.

\subsection{Analytically klt and log canonical singularities}

We first extend the notion of klt pair and log canonical pair from the algebraic
to the analytic category.

\begin{defn}[Analytically klt and log canonical pairs]\label{def:klt-lc}
  Let $X$ be a normal complex space and $D$ an effective $\bQ$-divisor on
  $X$. Let $p \in X$. We call the pair $(X, D)$ \emph{analytically log canonical
    at p}, respectively \emph{analytically klt at p} if there exists an open
  neighbourhood $U = U(p)$ such that
  \begin{enumerate}
  \item \label{il:klt1} $K_X + D$ is $\bQ$-Cartier on $U$,
  \item \label{il:klt2} if $\pi: \wtilde U \to U$ is any log resolution of
    $(U,D|_U)$, and $m \in \bN$ is an integer such that $m(K_X +D)|_U = m(K_U +
    D|_U)$ is Cartier on $U$, then writing
    \begin{multline*}
      \qquad\qquad\qquad m(K_{\wtilde U} + \pi^{-1}_*(D|_U))
      \sim_{\bQ-\text{lin.}} \\ \pi^* \bigl( m(K_U +D|_U) \bigr) \;+\;
      \sum_{\mathclap{\txt{\scriptsize $E$ a $\pi$-except.\\ \scriptsize
            divisor}}} m \cdot a(E, U, D|_U) \cdot E
    \end{multline*}
    with rational numbers $ a(E, U, D|_U) \in \bQ$, we have $a(E, U, D_U) \geq
    -1$, respectively $a(E, U, D|_U) > -1$ for all $\pi$-exceptional divisors
    $E$,
  \item \label{il:klt3}for the ``klt'' case, $\lfloor D|_U \rfloor = 0$.
  \end{enumerate}
  A pair $(X, D)$ is called \emph{analytically log canonical}, respectively
  \emph{analytically klt} if it is analytically log canonical, respectively
  analytically klt at every point $p \in X$.
\end{defn}

\begin{rem}
  Resolutions of singularities, as well as log-resolutions, are known to exist
  in the analytic category, cf.~\cite[Thm.~2.0.1]{MR2500573} and
  \cite[Thm.~7.1]{MR0499299}.
\end{rem}

\begin{rem}
  The notion of an ''analytically terminal'' pair is defined in complete analogy
  to Definition~\ref{def:klt-lc} above. A pair $(X, D)$ is analytically klt at a
  point $p \in X$, respectively analytically log canonical at a point $p \in X$
  if and only if the triple $(X, \{p\}, D)$ is klt, respectively log canonical
  in the sense of \cite[p.~104]{Kawa88}.
\end{rem}

\begin{rem}
  As in the algebraic case, the discrepancy condition can be checked on a single
  resolution. More precisely, a pair $(X, D)$ is analytically klt, respectively
  analytically log canonical at a given point $p \in X$ if and only if there
  exists an open neighbourhood $U$ of $p$ such that
  (\ref{def:klt-lc}.\ref{il:klt1}) and (\ref{def:klt-lc}.\ref{il:klt3}) hold,
  and the discrepancy inequalities required in (\ref{def:klt-lc}.\ref{il:klt2})
  are fulfilled for a single log resolution $\pi: \wtilde U \to U$.
\end{rem}

\begin{defn}[Non-klt locus]
  Let $(X, D)$ be an analytically log canonical pair. Then we define the
  \emph{non-klt locus of }$(X,D)$ to be
  $$
  \nklt(X,D) := \{p \in X \mid (X,D) \text{ is not klt at }p \}.
  $$
\end{defn}

\begin{rem}
  Let $(X,D)$ be an analytically log canonical pair. Then, $\nklt(X,D)$ is a
  Zariski-closed subset of $X$.
\end{rem}

Next, we compare the newly introduced classes of analytic singularities to the
classical notions in the algebraic category. In the following, given an
algebraic object $O$ the corresponding analytic object will be denoted by
$O^{an}$.

\begin{lem}\label{lem:kltcomparison}
  Let $X$ be a normal algebraic variety and $D$ an effective $\bQ$-divisor on $X$. Then
  \begin{enumerate-p}
  \item\label{il:kltcomparison1} $(X,D)$ is log canonical if and only if
    $(X^{an},D^{an})$ is analytically log canonical.
  \item\label{il:kltcomparison2} $(X,D)$ is klt if and only if $(X^{an},D^{an})$
    is analytically klt.
  \item\label{il:kltcomparison3} If $(X,D)$ is log canonical, we have
    $$
    \nklt(X^{an},D^{an}) = (\nklt(X,D))^{an}.
    $$
  \end{enumerate-p}
\end{lem}
\begin{proof}
  This follows from the fact that whether or not $(X,D)$ is log canonical or klt
  can be checked on open subsets of $(X^{an},D^{an})$, see \cite[p.~104]{Kawa88}.
\end{proof}

\begin{defn}\label{defn:locallyalgebraic}
  Let $X$ be a normal complex space, and $D$ an effective $\bQ$-divisor on
  $X$. We call the pair $(X,D)$ \emph{locally algebraic} if there exists a cover
  $\{U_\lambda\}_{\lambda \in \Lambda}$ of $X$ by open subsets $U_\lambda$ such
  that for every $\lambda \in \Lambda$ there exists
  \begin{itemize}
  \item a normal quasi-projective variety $Y_\lambda$,
  \item an effective $\bQ$-divisor $D_\lambda$ on $Y_\lambda$,
  \item  an open subset $V_\lambda \subset (Y_\lambda)^{an}$, and 
  \item a biholomorphic map of pairs $\phi_\lambda: (U_\lambda, D|_{U_\lambda}) \to (V_\lambda, (D_\lambda)^{an}|_{V_\lambda})$. 
  \end{itemize}
\end{defn}

\begin{example}
  Every complex manifold and every complex space with finite quotient
  singularities is locally algebraic.
\end{example}

\begin{example}\label{eq:Artin}
  It follows from Artin's Approximation Theorem \cite[Thm.~3.8]{ArtinApprox}
  that any complex space with isolated singularities is locally algebraic.
\end{example}

\begin{example}
  Let $X$ be a Moishezon space. Then, $X$ is the complex space associated with
  an algebraic space $Z$ in the sense of Artin, see~\cite[Theorem 7.3]{ArtinII}.
  Consequently, every point $p \in X=Z^{an}$ admits an étale neighbourhood
  $\varphi: U \to Z$, where $U$ is an affine algebraic variety. It follows that
  $X$ is locally algebraic in the sense of
  Definition~\ref{defn:locallyalgebraic} above.
\end{example}

\subsection{The extension theorem for holomorphic forms}

If $Z$ is any non-singular complex space and $D$ is a normal crossing divisor on
$Z$, then the sheaf of holomorphic (logarithmic) $p$-forms will be denoted by
$\Omega_Z^p(\log D)$.

The following is the main result of this section.

\begin{thm}[Extension theorem for holomorphic forms]\label{thm:holomorphicextension}
  Let $X$ be a locally algebraic normal complex space and $D$ an effective
  $\bQ$-divisor on $X$ such that $(X, D)$ is analytically log canonical. Let
  $\pi: \wtilde X \to X$ be a log resolution of $(X, D)$ with $\pi$-exceptional
  set $E$ and
  $$
  \wtilde D := \text{largest reduced divisor contained in } \supp
  \pi^{-1}(\nklt(X,D)).
  $$
  Then, the sheaves $\pi_*\Omega_{\wtilde X}^p(\log \wtilde D) $ are reflexive
  for all $p \leq n$.
\end{thm}

\begin{rem}
  As in the algebraic case, the name ''Extension Theorem'' is justified by the
  following observation: the sheaves $\pi_*\Omega_{\wtilde X}^p(\log \wtilde D)
  $ are reflexive if and only if for any open set $U \subseteq X$ and any number
  $p$, the restriction morphism
  \begin{equation*}
    H^0\bigl(U,\, \pi_*\Omega^p_{\wtilde X}(\log \wtilde D)\bigr) \to
    H^0\bigl(U \setminus \pi(E),\, \Omega^p_{X}(\log \lfloor D\rfloor)\bigr)
  \end{equation*}
  is surjective. In other words, Theorem~\ref{thm:holomorphicextension} states
  that any holomorphic logarithmic $p$-form defined on the non-singular part of
  $(X, D)$ can be extended to any resolution of singularities, possibly with
  logarithmic poles along the divisor $\wtilde D$.
\end{rem}

\begin{rem}
  Note that Theorem~\ref{thm:anaext} as stated in the introduction
  (Section~\ref{sect:intro}) of this paper is an immediate consequence of
  Theorem~\ref{thm:holomorphicextension} ($\nklt(X, \emptyset)$ being empty in
  this case).
\end{rem}

\subsection{Proof of Theorem~\ref{thm:holomorphicextension}}

We begin with two preparatory lemmas.

\begin{lem}\label{lem:localise}
  Let $X$ be a normal complex space and $D$ an effective $\bQ$-divisor on $X$
  such that $(X, D)$ is analytically log canonical. Let $\pi: \wtilde X \to X$
  be a log resolution of $(X, D)$ and
  $$
  \wtilde D := \text{largest reduced divisor contained in } \supp
  \pi^{-1}(\nklt(X,D)).
  $$
  Then, the sheaf $\pi_*\Omega_{\wtilde X}^p(\log \wtilde D) $ is reflexive if
  and only if there exists an open cover $\{U_\lambda\}_{\lambda \in \Lambda}$
  of $X$ as well as log resolutions $\pi_\lambda: \wtilde{U}_\lambda \to
  U_\lambda$ such that for every $\lambda \in \Lambda$ the sheaf
  $(\pi_\lambda)_*\Omega_{\wtilde{U}_\lambda}^p(\log \wtilde{D}_\lambda) $ is
  reflexive, where
  $$
  \wtilde{D}_\lambda := \text{largest reduced divisor contained in } \supp
  (\pi_\lambda)^{-1}(\nklt(U_\lambda ,   D|_{U_\lambda})). 
  $$
\end{lem}
\begin{proof}
  Log resolutions and principalization morphisms exist in the analytic category,
  cf.~\cite[Thms.~2.0.1 and 2.0.3]{MR2500573} and
  \cite[Sect.~7]{MR0499299}. This immediately implies that any two log
  resolutions of an analytic pair $(Z, \Delta)$ can be dominated by a third. The
  claim of Lemma~\ref{lem:localise} thus follows by arguments analogous to the
  algebraic case proven for example in \cite[Lem.~2.13]{GKK08}.
\end{proof}

\begin{lem}\label{lem:reflexivecomparison}
  Let $\pi: Y \to X$ be a projective morphism of normal algebraic varieties, and
  let $\sF$ be a coherent algebraic sheaf on $Y$. If $\pi_*\sF$ is a reflexive
  sheaf on $X$, then $(\pi^{an})_*(\sF^{an})$ is an (analytically) reflexive
  sheaf on $X^{an}$.
\end{lem}
\begin{proof}
  Relative GAGA \cite[Ch.~XII, Thm.~4.2]{SGA1} applied to the projective
  morphism $\pi$ yields
  \begin{equation}\label{eq:456}
    (\pi^{an})_*(\sF^{an}) \cong (\pi_*\sF)^{an}.
  \end{equation}
  Moreover, for any coherent algebraic sheaf $\sG$ on $X$ we have
  \begin{equation}\label{eq:homandhomagain}
    \sHom_{\sO_X}(\sG, \sO_X)^{an} \cong \sHom_{\sO_{X^{an}}}(\sG^{an}, \sO_{X^{an}}).
  \end{equation}
  To prove~\eqref{eq:homandhomagain}, consider the natural morphism of ringed
  spaces $\varphi: X^{an} \to X$. Recall from \cite[Exposé~XII, Sect.~1.1]{SGA1}
  that $\varphi$ is flat, and that
  $$
  \sG^{an} = \varphi^* \sG \text{\quad and \quad} \sO_{X^{an}} = \varphi^* \sO_X.
  $$
  Recalling further that $\sG$ is finitely presented, \cite[Chap.~0,
  Sect.~5.3.2]{EGA1}, Equation~\eqref{eq:homandhomagain} is an immediate
  consequence of \cite[Chap.~0, Sect.~6.7.6]{EGA1}.

  Applying the functor $\sHom$ twice, Equation~\eqref{eq:homandhomagain} shows
  that the analytification of any reflexive coherent algebraic sheaf on $X$ is
  an (analytically) reflexive coherent analytic sheaf on $X^{an}$. Since
  $\pi_*\sF$ is reflexive by assumption, the claim hence follows from the
  isomorphism \eqref{eq:456}.
\end{proof}

\begin{proof}[Proof of Theorem~\ref{thm:holomorphicextension}]
  Since $X$ is locally algebraic, using Lemma~\ref{lem:localise} we may assume
  that there exists a quasi-projective variety $Z$ and an effective
  $\bQ$-divisor $\Delta$, a log resolution $q: \wtilde {Z} \to Z$ of $(Z,
  \Delta)$, as well as open holomorphic embeddings $\varphi: (X,D)
  \hookrightarrow (Z^{an},\Delta^{an})$ and $\psi: \wtilde X \hookrightarrow
  \wtilde{Z}^{an}$ such that the following diagram commutes
  $$
  \begin{xymatrix}{ %
      \wtilde X \ar[r]^{\psi} \ar[d]_\pi&   \wtilde{Z}^{an} \ar[d]^{q^{an}}\\
      X \ar[r]^{\varphi} &  Z  ^{an}.
    }
  \end{xymatrix}
  $$
  
  Lemma~\ref{lem:kltcomparison}.(\ref{il:kltcomparison1}) implies that
  $(Z,\Delta)$ is log canonical at every point of $\varphi(X)$. Hence, removing
  a closed algebraic subset from $Z$ if necessary, we may assume that $(Z,
  \Delta)$ is log canonical. Furthermore,
  Lemma~\ref{lem:kltcomparison}.(\ref{il:kltcomparison2}) implies
  that
  $$
  \nklt(X,D) = \varphi^{-1}(\nklt(Z, \Delta)^{an}).
  $$
  As a consequence, setting
  $$
  \wtilde{\Delta} := \text{largest reduced divisor contained in } \supp
  \pi^{-1}(\nklt(Z,\Delta)),
  $$
  we have $\wtilde D = \psi^{-1}(\wtilde{\Delta}^{an})$. In order to establish
  the claim it therefore suffices to show that $(q^{an})_*
  \Omega_{\wtilde{Z}^{an}}^p (\log \wtilde{\Delta}^{an}) $ is reflexive.

  Elementary considerations show that for all $p \leq n$ 
  \begin{equation}\label{eq:123}
    \Omega_{\wtilde{Z}^{an}}^p (\log \wtilde{\Delta}^{an}) = (\Omega_{\wtilde{Z}}^p(\log \wtilde{\Delta}))^{an}.
  \end{equation}
  Since $q_*\Omega_{\wtilde{Z}}^p(\log \wtilde{\Delta})$ is reflexive by the
  Extension Theorem, \cite[Thm.~1.5]{GKKP11},
  Lemma~\ref{lem:reflexivecomparison} together with equation~\eqref{eq:123}
  yields reflexivity of $(q^{an})_* \Omega_{\wtilde{Z}^{an}}^p (\log
  \wtilde{\Delta}^{an})$. This finishes the proof.
\end{proof}

The observation made in Example~\ref{eq:Artin} immediately yields the following
consequence of Theorem~\ref{thm:holomorphicextension}.

\begin{cor}[Holomorphic extension over isolated singularities]\label{cor:isolated}
  Let $X$ be a normal complex space, $D$ a $\bQ$-divisor on $X$ such that $(X,
  D)$ has only isolated log canonical singularities. Let $\pi: \wtilde X \to X$
  be a log resolution of $(X,D)$ and
  $$
  \wtilde D := \text{largest reduced divisor contained in } \supp
  \pi^{-1}(\nklt(X,D)).
  $$
  Then, the sheaves $\pi_*\Omega_{\wtilde X}^p(\log \wtilde D)$ are reflexive
  for all $p \leq n$.
\end{cor}

\begin{rem}
  In particular, Corollary~\ref{cor:isolated} implies that any holomorphic
  differential forms defined on the smooth locus of a three-dimensional complex
  space $X$ such that $(X, \emptyset)$ is analytically terminal extends to any
  resolution of singularities.
\end{rem}

\section{Reflexive differential forms on rationally connected varieties}
\label{sect:rcc}

\subsection{Non-existence of reflexive pluri-forms}

Rationally connected and, more generally, rationally chain-connected projective
varieties play a key role in the classification and structure theory of
algebraic varieties. It is a fundamental fact of higher-dimensional algebraic
geometry that a rationally connected projective manifold $X$ does not carry any
pluri-forms, that is
\begin{equation}\label{eq:nopluriforms}
  H^0 \bigl(X,\,(\Omega^1_X)^{\otimes m} \bigr) = 0 \quad \forall m \in \bN^+.
\end{equation}
We refer to \cite[IV.3.8]{K96} for a thorough review of this result.  At least
conjecturally, \eqref{eq:nopluriforms} is also sufficient to conclude that a
projective manifold $X$ is projective. This has been proven in dimension three
by Kollár--Miyaoka--Mori, see \cite[Thm.~3.2]{KMM92b}.

For reflexive $p$-forms, the vanishing result has been generalized to spaces
which support klt pairs.

\begin{thm}[\protect{Reflexive differentials on rcc spaces, cf.~\cite[Thm.~5.1]{GKKP11}}]\label{thm:nopformsonrcc}
  Let $X$ be a normal, rationally chain-connected projective variety. If there
  exists a $\bQ$-divisor $D$ on $X$ such that $(X,D)$ is klt, then $H^0 \bigl(
  X, \, \Omega^{[p]}_X \bigr) = 0$ for all $1 \leq p \leq \dim X$. \qed
\end{thm}

\begin{rem}[Rational chain-connectedness vs.\ rational connectedness]\label{rem:rcc}
  Let $X$ be a normal, rationally chain-connected projective variety. If there
  exists a $\bQ$-divisor $D$ on $X$ such that $(X,D)$ is klt, then $X$ is in
  fact rationally connected, cf.~\cite[Cor.~1.5]{HMcK07}.
\end{rem}

In this section we investigate whether a similar vanishing result also holds for
reflexive pluri-differentials. Somewhat surprisingly, the answer is mixed:
vanishing holds if $X$ is factorial, but fails in general.

\begin{thm}[Reflexive pluri-forms on factorial rcc spaces]\label{thm:rc}
  Let $X$ be a normal rationally chain-connected projective variety. If $X$ is
  factorial and has canonical singularities, then
  $$
  H^0 \left(X,\,(\Omega^1_X)^{[m]} \right) = 0 \quad \text{for all } m \in
  \bN^+, \text{ where } (\Omega^1_X)^{[m]} := \bigl( (\Omega^1_X)^{\otimes
    m}\bigr)^{**}.
  $$
\end{thm}

\begin{rem}[Relation between Theorems~\ref{thm:nopformsonrcc} and \ref{thm:rc}]
  Let $X$ be a normal space. Assume that there exists a $\bQ$-divisor $D$ on $X$
  such that $(X,D)$ is klt. If $X$ is factorial, then $X$ has canonical
  singularities, cf.~\cite[Cor.~2.35]{KM98}.
\end{rem}

\begin{rem}[Necessity of the assumption that $X$ is canonical]
  There are examples of rational surfaces $X$ with log terminal singularities
  whose canonical bundle is torsion or even ample,
  cf.~\cite[Section 10]{MR2931877} or
  \cite[Example~43]{MR2400877}. Since $H^0\bigl(X,\, \sO_X(mK_X) \bigr) \subset
  H^0\bigl(X,\, (\Omega_X^1)^{[m \cdot \dim X]} \bigr)$, these examples show
  that the assumption that $X$ has \emph{canonical} singularities cannot be
  omitted in Theorem~\ref{thm:rc}.
\end{rem}

\begin{rem}[Sharpness of Theorem~\ref{thm:rc}]
  In the class of varieties with canonical singularities Theorem~\ref{thm:rc} is
  sharp: Section~\ref{ssec:C1} contains an example of a rationally connected
  $\bQ$-factorial (but non-factorial) surface with canonical singularities which
  does carry non-trivial reflexive pluri-forms.
\end{rem}

The proof of Theorem~\ref{thm:rc} uses the notion of semistable sheaves on
singular spaces, where semistability is defined respect to a movable curve
class. Appendix~\ref{sec:snsp} contains detailed proofs of the necessary
foundational results used here.

\begin{proof}[Proof of Theorem~\ref{thm:rc}]
  We argue by contradiction and assume that there exists a positive number $m
  \in \bN$ and a non-trivial element
  $$
  \sigma \in H^0 \left( X,\, (\Omega^1_X)^{[m]} \right) \setminus \{ 0 \}.
  $$
  Let $\pi: \wtilde X \to X$ be a log resolution of $X$ with exceptional divisor
  $E$. By Remark~\ref{rem:rcc} the smooth variety $\wtilde X$ is rationally
  connected. As explained in \cite[IV.3.9.4]{K96}, we can therefore choose a
  dominating family of smooth rational curves in $\wtilde X$, and a general
  member $\wtilde C \cong \bP^1$ such that the following holds.
  \begin{equation}\label{eq:vielText}
    \text{The restriction of $\sT_{\wtilde X}$ to the smooth curve $\wtilde C$ is an ample vector bundle.}
  \end{equation} 
  By general choice, the curve $\wtilde C$ will not be contained in the
  $\pi$-exceptional set $E$.  We consider the image curve $C := \pi(\wtilde C)$
  with its reduced structure, and write $\alpha = [C]$ for the corresponding
  class in the Mori-cone $\overline{\NE}(X)$. Since $[\wtilde C]$ is movable,
  that is, $[\wtilde C]$ has non-negative intersection with all effective
  divisors, the class $[C]$ is movable as well.

  The section $\sigma$ defines an inclusion $\sO_X \subseteq
  (\Omega^1_X)^{[m]}$. Using the notation $\mu_\alpha^{\max}$ as introduced in
  Definition~\ref{def:semistability}, we see that $\mu_\alpha^{\max}
  \bigl((\Omega^1_X)^{[m]} \bigr) \geq 0$. It follows from
  Item~(\ref{prop:SSreflTensor}.\ref{il:p161}) of
  Proposition~\ref{prop:SSreflTensor} that
  $\mu_\alpha^{\max}\bigl(\Omega^{[1]}_X \bigr) \geq 0$.
  Proposition~\ref{prop:mumax} therefore allow us to find a subsheaf $\sS
  \subseteq \Omega_X^{[1]}$ of rank $r \geq 1$ satisfying
  \begin{equation}\label{eq:slopeofS}
    \mu_\alpha(\sS) = \mu_\alpha^{\max}\bigl(\Omega^{[1]}_X \bigr) \geq 0.
  \end{equation}
  Since $X$ is supposed to be factorial, the determinant $\det \sS$ is an
  invertible subsheaf of $\Omega_X^{[r]}$. Inequality \eqref{eq:slopeofS}
  implies that the restriction $(\det \sS)|_C$ is a nef line bundle on the
  (possibly singular) curve $C$. It follows that $\pi^*(\det \sS)|_{\wtilde C} =
  (\pi|_{\wtilde C})^*\bigl((\det \sS)|_C\bigr)$ is a nef line bundle on
  $\wtilde C$.
  
  Finally, recall from \cite[Thm.~4.3]{GKKP11} that there exists a pull-back map
  $\phi:~\pi^*\Omega^{[r]}_X \to \Omega^r_{\wtilde X}$, isomorphic away from
  $E$. Its restriction to $\wtilde C$ induces the following sequence of sheaf
  morphisms
  \begin{equation}\label{eq:sequence}
    \xymatrix{ %
      \pi^*(\det \sS)|_{\wtilde C} \ar[r] \ar@/^6mm/[rr]^{\psi := \text{ composition}} &
      \pi^*\Omega^{[r]}_X\bigr|_{\wtilde C} \ar[r]_{\phi} &
      \Omega^r_{\wtilde X}\bigr|_{\wtilde C}.
    }
  \end{equation}
  Since $\wtilde C$ is not contained in $E$, all arrows in~\eqref{eq:sequence}
  are generically injective; in particular, $\psi$ is not the trivial map.

  On the other hand, note that the restriction $\Omega^r_{\wtilde
    X}\bigr|_{\wtilde C}$ is negative owing to \eqref{eq:vielText}. Since both
  $\pi^*(\det \sS)|_{\wtilde C}$ and $\Omega^r_{\wtilde X}|_{\wtilde C}$ are
  locally free, and $\psi$ is non-trivial, it is an injection of a nef line
  bundle into a negative vector bundle, which is absurd. This contraction
  concludes the proof.
\end{proof}

\subsection{A counterexample in the $\bQ$-factorial setup}
\label{ssec:C1}

If $X$ is not factorial, the arguments used in the proof of Theorem~\ref{thm:rc}
break down: with the notation used in the proof, we can write
$$
(\pi|_{\wtilde C})^*(\det \sS) \cong \sT \oplus \sA
$$
where $\sT $ is torsion and $\sA$ locally free of rank one. The argument fails
because $\sA$ might be negative. The following example shows that
Theorem~\ref{thm:rc} is in fact no longer true in the non-factorial setting,
even if $X$ has the mildest form of $\bQ$-factorial singularities.

\begin{example}[A rationally connected surface supporting pluri-differential forms]\label{ex:rcowpf}
  \begin{figure}[t]
    \centering
    
    $$
    \xymatrix{
      \protect{
        \begin{tikzpicture}
          \draw (0,0) to [out=65, in=-65] (0,3.5);
          \draw (0.05,0) to [out=65, in=-65] (0.05,3.5);
          \fill (0.16,2.5)  node[right]{$\ast$ \scriptsize $A_1$ point};
          \fill (0.16,1.0)  node[right]{$\ast$ \scriptsize $A_1$ point};
        \end{tikzpicture}
      }
      &&&
      \protect{
        \begin{tikzpicture}
          \draw (0.5, 0.0) to [out=100, in=-50] (0.0, 1.4);
          \draw (0.0, 1.0) to [out=65, in=-65] (0.0, 2.5);
          \draw (0.05, 1.0) to [out=65, in=-65] (0.05, 2.5);
          \draw (0.0, 2.1) to [out=50, in=-100] (0.5, 3.5);
          \fill (0.2, 1.75) node[right]{\scriptsize $(-1)$};
          \fill (0.5, 3.0) node[left]{\scriptsize $(-2)$};
          \fill (0.5, 0.5) node[left]{\scriptsize $(-2)$};
        \end{tikzpicture}
      }
      \ar[rr]^{\text{blow-up}}_{\text{of }x'}
      \ar[lll]_{\text{contraction of}}^{(-2)-\text{curves}} &&
      \protect{
        \begin{tikzpicture}
          \draw (0.2, 0.0) to [out=65, in=-50] (0.0, 2.2);
          \draw (0, 1.3) to [out=50, in=-65] (0.2, 3.5);
          \fill (0.28,1.75) circle (2pt) node[right]{\scriptsize $x'$};
          \fill (-0.4, 3.0)  node[right]{\scriptsize $(-1)$};
          \fill (-0.4, 0.5)  node[right]{\scriptsize $(-1)$};
        \end{tikzpicture}
      }
      \ar[rr]^{\text{blow-up}}_{\text{of }x} &&
      \protect{
        \begin{tikzpicture}
          \draw (0,0) to [out=65, in=-65] (0,3.5);
          \fill (0.43,1.75) circle (2pt) node[right]{\scriptsize $x$};
        \end{tikzpicture}
      }
    }
    $$
    
    \caption{Birational Transformations Used in Example~\ref{ex:rcowpf}}
    \label{fig:btrans}
  \end{figure}
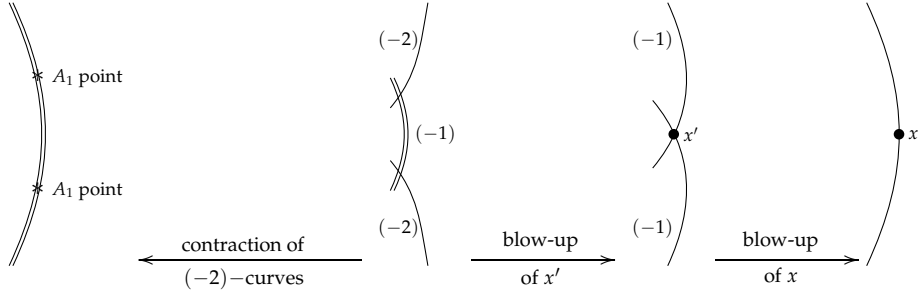

  Let $\pi' : X' \to \bP^1$ be any rational ruled surface. Choose four distinct
  points $q_1, \ldots, q_4$ in $\bP^1$. For each point $q_i$, perform the
  following sequence of birational transformations of the ruled surface,
  outlined also in Figure~\ref{fig:btrans}.
  \begin{enumerate}\label{en:blowups}
  \item Choose a point $x$, contained in the fibre over $q_i$, and blow up this
    point. The result is a surface with a map to $\bP^1$ such that the fibre
    over $q_i$ is the union of two reduced rational curves each with
    self-intersection number $(-1)$, meeting transversely in a point $x'$.
    
  \item\label{il:step2} Blow up the point $x'$. The result is a surface with a
    map to $\bP^1$ such that the fibre over $q_i$ is the union of two reduced
    rational curves each with self-intersection number $(-2)$, and one rational
    curve with self-intersection $(-1)$. The $(-2)$-curves are disjoint, the
    $(-1)$-curve appears in the fibre with multiplicity two.

  \item Blow down the $(-2)$-curves contained in the fibre over $q_i$. The
    result is a normal surface with a map to $\bP^1$ such that the set-theoretic
    fibre $F_i$ over $q_i$ is a smooth rational curve.  The curve $F_i$ appears
    in the cycle theoretic fibre over $q_i$ with multiplicity two. The surface
    has two singular points of type $A_1$ on $F_i$. Seen as a divisor, the curve
    $F_i$ is not Cartier, but $2\cdot F_i$ is.
  \end{enumerate}
  This way, we obtain a rational, rationally connected surface $\pi : X \to
  \bP^1$ containing eight singular points, two on each of the fibres $F_1,
  \ldots, F_4$.

  We claim that there exist sections in the second reflexive product $\bigl(
  \Omega^1_X \bigr)^{[2]}$. To this end, let $X^\circ \subset X$ be the smooth
  locus of $X$, and set $F_i^\circ := X^\circ \cap F_i$, for each $1 \leq i \leq
  4$. Finally, choose any point $p \in \bP^1 \setminus \{q_1, \ldots, q_4\}$ and
  let $F := \pi^{-1}(p)$ denote the associated fibre. Since the fibres over the
  $q_i$ all have multiplicity two, it is clear that the pull-back map of
  differentials,
  $$
  d(\pi|_{X^\circ}) : \underbrace{\pi^*(\Omega^1_{\bP^1})|_{X^\circ}}_{\cong
    \mathcal O_{X^\circ}(-2\cdot F)} \to \Omega^1_{X^\circ},
  $$
  has single zeros along the $F_i^\circ$. Accordingly, there exists a
  factorization
  \begin{equation}\label{eq:fact}
    \xymatrix{ %
      \pi^*(\Omega^1_{\bP^1})|_{X^\circ} \ar[r]
      \ar@/^.6cm/[rr]^{d(\pi|_{X^\circ})} & \mathcal O_{X^\circ}(-2\cdot F +
      \sum_i F_i^\circ) \ar[r]_(.77){\alpha} &
      \Omega^1_{X^\circ}. %
    }
  \end{equation}
  Setting $D := -2\cdot F + \sum_i F_i$, Factorization~\eqref{eq:fact} then
  gives a non-trivial morphism of reflexive sheaves
  $$
  \alpha^{[1]} : \mathcal O_X(D) \to \Omega^{[1]}_X, \quad \text{hence also}
  \quad \alpha^{[2]} : \mathcal O_X(2\cdot D) \to \bigl(\Omega^1_X\bigr)^{[2]} =
  \bigl(\Omega^1_X \otimes \Omega^1_X \bigr)^{**}.
  $$
  To finish the construction, observe that the divisor $D$ is $\bQ$-linearly
  equivalent to zero, but not linearly equivalent to zero. The divisor $2\cdot
  D$, however, is linearly equivalent to zero, giving a non-trivial map
  \begin{equation}\label{eq:Schnitt}
    \underbrace{H^0\bigl( X,\, \mathcal O_X(2\cdot D) \bigr)}_{\cong \bC}
    \to H^0\bigl( X,\, (\Omega^1_X)^{[2]} \bigr) \not = 0,
  \end{equation}
  that corresponds to a non-trivial section $\tau \in H^0\bigl( X,\,
  (\Omega^1_X)^{[2]} \bigr)$.
\end{example}

\begin{remq}\label{remq}
  There are other ways to see that the surface $X$ constructed in
  Example~\ref{ex:rcowpf} admits a pluri-form. Semistable reduction yields a
  diagram
  $$
  \xymatrix{ %
    \wtilde X \ar[rrrr]^{\text{2:1 cover}}_{\text{branched over the singularities}} \ar[d]_{\text{$\bP^1$-bundle}}^{\wtilde \pi} &&&& X \ar[d]^{\pi} \\
    E \ar[rrrr]^{\text{2:1 cover}}_{\text{branched over $q_1, \ldots, q_4$}} &&&& \bP^1 }
  $$
  where $\wtilde X$ is smooth, $E$ is an elliptic curve, and where the vertical
  arrows are quotients by the associated action of $G := \factor \bZ.2\bZ.$. It
  is not difficult to construct a $G$-invariant form
  $$
  \wtilde \tau \in H^0\bigl( \wtilde X,\, (\wtilde \pi^* \Omega^1_E)^{\otimes
    2}\bigr) \subset H^0\bigl( \wtilde X,\, (\Omega^1_{\wtilde X})^{\otimes
    2}\bigr)
  $$
  that corresponds to the form $\tau$ constructed above. One may ask if any
  pluri-form on a rationally connected space arises in one of the two ways
  indicated in our specific example.
\end{remq}

\begin{rem}
  Example~\ref{ex:rcowpf} underlines the fact that the extension theorem is not
  valid for pluri-forms, cf.~\cite[Example~3.1.3]{GKK08}.
\end{rem}

\section{Vanishing theorems of Kodaira-Akizuki-Nakano type}
\label{sect:KAN}

Recall the statement of the Kodaira-Akizuki-Nakano Vanishing Theorem in the
smooth setting.

\begin{thm}[Kodaira-Akizuki-Nakano Vanishing Theorem, \cite{Akizuki-Nakano54}]\label{thm:KANvanishing}
  Let $X$ be a smooth projective variety and let $\sL$ be an ample line bundle
  on $X$. Then
  \begin{align}
    H^q \bigl(X,\,\Omega^{p}_X \otimes \sL \bigr) & = 0 \quad \text{for $p + q
      > n$, and}\label{eq:KANv1} \\
    H^q \bigl(X,\,\Omega^{p}_X \otimes \sL^{-1} \bigr) & = 0 \quad \text{for $p
      + q < n$.}\label{eq:KANv2}
  \end{align}
\end{thm}

\begin{rem}\label{rama}
  Assertions~\eqref{eq:KANv1} and \eqref{eq:KANv2} are equivalent to one another
  by Serre duality. Ramanujam \cite{Ramanujam72} gave a simplified proof of
  Theorem~\ref{thm:KANvanishing} and showed that it does not hold if one only
  requires $\sL$ to be semi-ample and big.
\end{rem}

Esnault and Viehweg generalised Theorem~\ref{thm:KANvanishing} to logarithmic
differentials, \cite{MR853449}. Moreover, Kodaira-Akizuki-Nakano vanishing has
been shown to hold for sheaves of reflexive differentials on varieties with
quotient singularities, see~\cite{ArapuraKAN}, as well as on toric varieties,
see \cite[Thm.~9.3.1]{CoxLittleSchenck}.

In this section, we prove similar vanishing results for reflexive differentials
on varieties with more general singularities. However, these vanishing results
are restricted to special values of $p$ and $q$. It turns out that even for
spaces with isolated terminal Gorenstein singularities
Theorem~\ref{thm:KANvanishing} does not hold for arbitrary $p + q > n$,
respectively $p+q<n$. A corresponding example is provided in
Section~\ref{ssec:nopKANvan}.

\subsection{Partial vanishing results for lc and klt pairs}
\label{ssec:pKANvan}

In this section we prove some partial generalisations of
Theorem~\ref{thm:KANvanishing} to lc and klt pairs.

\begin{prop}[KAN-type vanishing for lc pairs, analogue of~\eqref{eq:KANv2}]\label{prop:partVan1}
  Let $X$ be a normal projective variety of dimension $n$, let $D$ be an
  effective $\bQ$-divisor on $X$ such that $(X,D)$ is log canonical, and let
  $\sL \in \Pic(X)$ be an ample line bundle.
  \begin{align}
    \label{eq:pv2}& H^0 \bigl(X,\,\Omega^{[p]}_X(\log \lfloor D \rfloor)
    \otimes \sL^{-1}\bigr) = 0 && \text{for all $p<n$,\,\,\, and} \\
    \label{eq:pv3}& H^1 \bigl(X,\,\Omega^{[p]}_X(\log \lfloor D \rfloor)
    \otimes \sL^{-1}\bigr) = 0 && \text{for all $p<n-1$.} 
  \end{align}
\end{prop}

\begin{rem}[Vanishing for dlt pairs]
  If $(X, D)$ is not only lc but dlt, then additionally $H^q \bigl(X,\,
  \sL^{-1}\bigr) = 0$ holds for all $q<n$. This follows by observing that $X$ is
  Cohen-Macaulay~\cite[Thm.~5.22]{KM98} and by using the more general result,
  \cite[Cor.~6.6]{KSS10}, that vanishing holds already if $(X,D)$ is lc and
  Cohen-Macaulay.
\end{rem}

\begin{proof}[Proof of Proposition~\ref{prop:partVan1}]
  First note that \eqref{eq:pv2} is a special case of the Bogomolov-Sommese
  vanishing theorem for log canonical pairs, \cite[Thm.~7.2]{GKKP11}.

  For the other case, choose a log resolution $\pi: \wtilde X \to X$, consider
  the set
  $$
  E := \bigl( \text{strict transform of } \supp \lfloor D \rfloor \bigr) \cup
  \bigl(\text{$\pi$-exceptional set}\bigr).
  $$
  and let $\sF_p := \Omega^p_{\wtilde X}(\log E) \otimes \pi^*\sL^{-1}$. The
  projection formula and the extension theorem \cite[Thm.~1.5]{GKKP11} together
  imply that $$\pi_* \sF_p = \Omega^{[p]}_X(\log \lfloor D \rfloor) \otimes
  \sL^{-1}.$$ In order to prove \eqref{eq:pv3} we need to show that $H^1
  \bigl(X,\,\pi_*\sF_p \bigr) = 0$ if $p<n-1$. For this, we will use the Leray
  spectral sequence, and Steenbrink vanishing,
  \cite[Thm.~2(a')]{Steenbrink85}. The latter asserts that
  \begin{equation}\label{eq:SBv}
    H^q \bigl(\wtilde X,\, \sF_p \bigr) = 0 \quad \text{for $p+q<n$.}
  \end{equation}

  Now consider the beginning of the five term exact sequence associated with the
  Leray spectral sequence,
  $$
  0 \to \underbrace{H^1\bigl( X,\, \pi_*\sF_p \bigr)}_{= E^{1,0}_2} \to
  \underbrace{H^1\bigl( \wtilde X,\, \sF_p \bigr)}_{= E^1} \to \underbrace{H^0
    \bigl( X,\, R^1\pi_*\sF_p \bigr)}_{= E^{0,1}_2} \to \cdots
  $$
  Steenbrink vanishing~\eqref{eq:SBv} gives that $E^1 = 0$, hence  $E^{1,0}_2 = 0$ 
  and therefore
  \eqref{eq:pv3}.
\end{proof}

The second proposition of this section partially generalizes
\eqref{eq:KANv1}. The authors would like to thank the referee for pointing out
that parts of the proposition also hold in the log canonical setting.

\begin{prop}[KAN-type vanishing for klt pairs, analogue of~\eqref{eq:KANv1}]\label{prop:partVan2}
  Let $X$ be a normal projective variety of dimension $n$, let $D$ be an
  effective $\bQ$-divisor on $X$, and let $\sL \in \Pic(X)$ be an ample line
  bundle. Then
  \begin{align}
    \label{eq:pv4} & H^q \bigl(X,\, \omega_X \otimes \sL \bigr) = 0 &&
    \text{for all $q>0$ if $(X,D)$ is klt, and} \\
    \label{eq:pv5} & H^n \bigl(X,\, \Omega^{[p]}_X \otimes \sL \bigr) = 0 &&
    \text{for all $p>0$ if $(X,D)$ is lc.}
  \end{align}
\end{prop}
\begin{proof}
  To prove~\eqref{eq:pv4}, choose a log resolution $\pi: \wtilde X \to X$. The
  extension theorem for differential forms on log canonical pairs,
  \cite[Thm.~1.5]{GKKP11}\footnote{or the fact that $X$ has rational
    singularities, \cite[Thm.~5.22]{KM98}} then asserts that $\omega_X = \pi_*
  \omega_{\wtilde X}$, and the assertion of \eqref{eq:pv4} is just the
  Grauert-Riemenschneider vanishing theorem \cite[p.~263]{GR70}.

  To prove~\eqref{eq:pv5}, consider the chain of isomorphisms,
  \begin{align*}
    H^n \bigl( X,\, \Omega^{[p]}_X \otimes \sL \bigr)^* & \cong \Hom
    \bigl(\Omega^{[p]}_X \otimes \sL,\, \omega_X \bigr) && \text{since
      $\omega_X$ is dualising, \cite[p.~241]{Ha77}}\\
    & \cong \Hom \bigl( \sL,\, \Omega^{[n-p]}_X \bigr) && \text{pairing
      assoc.~with wedge product.}
  \end{align*}
  The Bogomolov-Sommese vanishing theorem for log canonical pairs,
  \cite[Thm.~7.2]{GKKP11}, asserts that the last space is zero.
\end{proof}

\subsection{A counterexample to Kodaira-Akizuki-Nakano vanishing for klt spaces}
\label{ssec:nopKANvan}

We will show by way of example that the Kodaira-Akizuki-Nakano type vanishing
theorems of Propositions~\ref{prop:partVan1} and \ref{prop:partVan2} does not
hold for all values of $p$ and $q$, not even for Gorenstein spaces with isolated
terminal singularities.

\begin{example}[A fourfold with terminal singularities violating KAN vanishing]\label{ex:KANVisFalse}
  We construct a 4-dimensional terminal variety $X$ following the steps outlined
  in the following diagram.
  $$
  \xymatrix{ %
    \wtilde X \ar[d]_{\txt{\tiny $\pi$ \\\tiny contraction map}}
    \ar[rrrr]^{\psi, \text{ $\bP^1$-bundle}}_{\wtilde X := \bP_Y(\sO_Y(1)
      \oplus \sO_Y )} &&&& Y
    \ar[rrr]^{\phi, \text{ $\bP^1$-bundle}}_{Y := \bP_{\bP^2}(T_{\bP^2})} &&& \bP^2\\
    X }
  $$
  To describe the construction in detail, consider the 3-dimensional smooth
  variety $Y := \bP_{\bP^2} \bigl(\sT_{\bP^2} \bigr)$. The tangent bundle
  $\sT_{\bP^2}$ of the projective plane being ample, by definition the
  tautological bundle $\sO_Y(1) \in \Pic(Y)$ is ample. Better still, using the
  Euler sequence to present $\sT_{\bP^2}$ as a quotient of
  $\sO_{\bP^2}(1)^{\oplus 3}$, one shows that the ample bundle $\sO_Y(1)$ is in
  fact very ample.

  The bundle $\sO_Y(1)$ induces an embedding $Y \to \bP^N$, with projectively
  normal image. Let $X \subset \bP^{N+1}$ be the cone over $Y$. The variety $X$
  is then normal and has a single isolated singularity, the vertex $x \in
  X$. Blowing up $x$, we obtain a resolution of singularities, $\pi : \wtilde X
  \to X$. The variety $\wtilde X$ is isomorphic to the $\bP^1$-bundle $\psi:
  \bP_Y( \sO_Y(1) \oplus \sO_Y) \to Y$. The $\pi$-exceptional set $E$ is
  canonically identified with $\bP_{Y} \bigl( \sO_Y \bigr) \subseteq \wtilde
  X$. The divisor $E$ is thus a section of $\psi$ and naturally isomorphic to
  $Y$.  Its normal bundle is $N_{E/\wtilde X} \cong \sO_Y(-1)$. Finally,
  consider the ample bundle $\sL := \sO_X(1) \in \Pic(X)$.
\end{example}

The following two remarks summarise the main properties of $X$.

\begin{rem}[Dualising sheaves of $Y$, $X$ and $\wtilde X$]\label{rem:ex1}
  An elementary computation shows that the canonical bundle of $Y$ is given as
  \begin{equation}\label{eq:KY}
    \omega_Y \cong \sO_Y(-2)
  \end{equation}
  Using the bundle structure of $\wtilde X$, Equation~\eqref{eq:KY} and the
  adjunction formula, one computes the canonical bundle of $\wtilde X$ as
  \begin{equation}\label{eq:KX}
    \omega_{\wtilde X} \cong \sO_{\wtilde X}(E) \otimes \pi^* \sO_X(-3).
  \end{equation}
  Equation~\eqref{eq:KX} has two consequences:
  \begin{enumerate}
    \setcounter{enumi}{\value{equation}}
  \item\label{il:kabeljau} The dualising sheaf of $X$ is invertible, $\omega_X
    \cong \sO_X(-3) = \sO_{\bP^{N+1}}(-3)|_X$.
  \item The discrepancy formula for $\pi$ reads $K_{\wtilde X} =
    \pi^*(K_X)+E$.
  \end{enumerate}
  In particular, we obtain that the isolated singularity $x \in X$ is
  terminal. Recall from \cite[Thm.~5.22]{KM98} that terminal singularities are
  rational, hence Cohen-Macaulay. Assertion~(\ref{rem:ex1}.\ref{il:kabeljau})
  thus implies that $X$ is in fact Gorenstein.
\end{rem}

We are now ready to formulate the main results of this section. The following
two propositions show that both versions of the Kodaira-Akizuki-Nakano vanishing
theorem fail for the variety constructed in Example~\ref{ex:KANVisFalse}.

\begin{prop}[Generalisation of Proposition~\ref{prop:partVan1} does not hold]\label{prop:fail1}
  In the setup of Example~\ref{ex:KANVisFalse}, we have $H^2 \bigl(X,\,
  \Omega^{[1]}_X \otimes \sL^{-1} \bigr) \not = 0$.
\end{prop}

\begin{prop}[Generalisation of Proposition~\ref{prop:partVan2} does not hold]\label{prop:fail2}
  In the setup of Example~\ref{ex:KANVisFalse}, we have $H^2 \bigl( X,\,
  \Omega^{[3]}_X \otimes \sL \bigr) \not= 0$.
\end{prop}

The proofs of Proposition~\ref{prop:fail1} and \ref{prop:fail2}, both of which
rely on somewhat lengthy cohomology computations, are given in
Subsections~\ref{ssec:popf1} and \ref{ssec:popf2} below.

\subsubsection{KAN-vanishing and the Du~Bois complex}

Terminal singularities are rational, and therefore Du~Bois, see
\cite{Kovacs99}. By definition, this means that the zeroth graded piece of the
filtered de Rham (or Du~Bois) complex $\underline{\Omega}_X^\bullet$ is
quasi-isomorphic to $\sO_X$. One might wonder whether on terminal or more
generally log canonical spaces this remains true for higher degrees, that is,
whether or not the $p$-th graded piece of the Du~Bois complex is
quasi-isomorphic to $\Omega_X^{[p]}[-p]$, the complex having the single sheaf
$\Omega_X^{[p]}$ in the $p$-th place, for all values of $p$, cf.~\cite[Rem.~on
p.~180]{PetersSteenbrinkBook}. The Example~\ref{ex:KANVisFalse} constructed
above shows that this question has to be answered negatively.

\begin{prop}
  Let $X$ be the variety constructed in Example~\ref{ex:KANVisFalse}. Then, the
  $3$rd graded piece of the filtered de Rham complex of $X$ is not
  quasi-isomorphic to $\Omega_X^{[3]}[-3]$.
\end{prop}
\begin{proof}
  Denoting the filtered de Rham complex of $X$ by
  $(\underline{\Omega}_X^\bullet, F)$, the vanishing theorem of Guillen-Navarro
  Aznar-Puerta-Steenbrink \cite[Thm.~7.29]{PetersSteenbrinkBook} states that for
  any ample line bundle $\mathscr{L}$ on $X$ we have
  \begin{equation}\label{eq:Guillen}
    \mathbb{H}^m\bigl(X, \, \mathrm{Gr}_F^p \underline{\Omega}_X^\bullet \otimes \mathscr{L} \bigr) = 0 \quad \forall m > n.
  \end{equation}
  Suppose that $\mathrm{Gr}_F^3 \underline{\Omega}_X^\bullet$ is
  quasi-isomorphic to $\Omega_X^{[3]}[-3]$. Then \eqref{eq:Guillen} would imply
  that
  $$
  H^{2}\bigl(X, \Omega_X^{[3]} \otimes \mathscr{L}\bigr) \cong
  \mathbb{H}^{5}\bigl(X, \Omega_X^{[3]}[-3] \otimes \mathscr{L} \bigr) \cong
  \mathbb{H}^5\bigl(X, \, \mathrm{Gr}_F^3 \underline{\Omega}_X^\bullet \otimes
  \mathscr{L} \bigr) = 0,
  $$
  contradicting Proposition~\ref{prop:fail2} above.
\end{proof}

\subsubsection{Proof of Proposition~\ref*{prop:fail1}}
\label{ssec:popf1}

Proposition~\ref{prop:fail1} follows essentially from the Leray spectral
sequence. The following lemma summarises the relevant statements in our setting.

\begin{lem}\label{lem:isol}
  Let $X$ be a normal variety such that the pair $(X, \emptyset)$ is log
  canonical with isolated singularities. Assume furthermore that $\dim X \geq
  4$. Let $\pi: \wtilde X \to X$ be a log resolution of singularities, with
  $\pi$-exceptional divisor $E \subset \wtilde X$. If $\sL \in \Pic(X)$ is any
  ample line bundle, then
  \begin{equation}\label{eq:pisol1}
    H^2 \bigl( X,\, \Omega^{[1]}_X \otimes \sL^{-1} \bigr) \simeq H^0 \bigl( X,\,
    R^1\pi_* \, \Omega^1_{\wtilde X}(\log E) \bigr).
  \end{equation}
  In particular, it follows that the left hand side of~\eqref{eq:pisol1} is
  independent of the ample line bundle $\sL$.
\end{lem}
\begin{proof}
  As in the proof of Proposition~\ref{prop:partVan1}, set $\sF_p :=
  \Omega^1_{\wtilde X}(\log E) \otimes \pi^*\sL^{-1}$ and consider the following
  excerpt from the exact five term exact sequence associated with the Leray
  spectral sequence,
  \begin{equation}\label{eq:pisol2}
    \underbrace{H^1\bigl( \wtilde X,\, \sF_p \bigr)}_{= E^1} \to
    \underbrace{H^0 \bigl( X,\, R^1\pi_*\sF_p \bigr)}_{= E^{0,1}_2} \to
    \underbrace{H^2\bigl( X,\, \pi_* \sF_p \bigr)}_{= E^{2,0}_2} \to \underbrace{H^2
      \bigl( \wtilde X,\, \sF_p \bigr)}_{= E^2}
  \end{equation}
  As before, Steenbrink vanishing \cite[Thm.~2(a')]{Steenbrink85} and the
  assumption that $\dim X \geq 4$ give that $E^1 = E^2 = 0$. Using the
  projection formula and the extension theorem \cite[Thm.~1.5]{GKKP11} to
  identify $\pi_* \sF_p$ with $\Omega^{[1]}_X\otimes \sL^{-1}$,
  Sequence~\eqref{eq:pisol2} reads
  $$
  H^2 \bigl( X,\, \Omega^{[1]}_X \otimes \sL^{-1} \bigr) = E^{2,0}_2 \cong
  E^{0,1}_2 = H^0\bigl( X,\, R^1\pi_* \, \Omega^1_{\wtilde X}(\log E) \otimes
  \sL^{-1} \bigr).
  $$
  Since $X$ has only isolated singularities, the push-forward sheaf $R^1\pi_*
  \Omega^1_{\wtilde X}(\log E)$ has finite support, and
  $$
  H^0\bigl( X,\, R^1\pi_* \, \Omega^1_{\wtilde X}(\log E) \otimes \sL^{-1}
  \bigr) \simeq H^0\bigl( X,\, R^1\pi_* \, \Omega^1_{\wtilde X}(\log E) \bigr),
  $$
  finishing the proof of Lemma~\ref{lem:isol}.
\end{proof}

\begin{proof}[Proof of Proposition~\ref{prop:fail1}]
  Lemma~\ref{lem:isol} asserts that it suffices to verify that the push-forward
  $R^1\pi_* \, \Omega^1_{\wtilde X}({\rm log}\ E)$, which is a skyscraper sheaf
  whose support equals the vertex $x \in X$, is not the zero sheaf. To this end,
  consider the residue sequence for logarithmic differentials,
  $$
  \xymatrix{ %
    0 \ar[r] & \Omega^1_{\wtilde X} \ar[r] & \Omega^1_{\wtilde X}(\log E)
    \ar[r] & \sO_E \ar[r] & 0. %
  }
  $$
  Recalling from the extension theorem that the inclusion $\pi_*
  \Omega^1_{\wtilde X} \to \pi_* \Omega^1_{\wtilde X}(\log E)$ is an
  isomorphism, the push-down yields an exact sequence
  \begin{equation}\label{eq:xs2}
    0 \to \underbrace{\pi_* \sO_E}_{= \bC_x} \to R^1\pi_*\Omega^1_{\wtilde X} \to R^1\pi_*
    \Omega^1_{\wtilde X}(\log E).
  \end{equation}
  In~\eqref{eq:xs2}, the symbol $\bC_x$ denotes the skyscraper sheaf supported
  on $x$ with stalk $\bC$.
 
  Next, observe that for any open neighborhood $U = U(x) \subseteq X$, the Chern
  classes
  $$
  c_1 \bigl( \psi^* \sO_Y(1) \bigr) \quad \text{and} \quad c_1 \bigl(
  (\phi\circ\psi)^* \sO_{\bP^2}(1) \bigr)
  $$
  yield two linearly independent elements of $H^1\bigl(\pi^{-1}(U),\,
  \Omega^1_{\wtilde X} \bigr)$, since their restrictions to $E\cong Y$ are
  linearly independent. The stalk of $R^1\pi_*\Omega^1_{\wtilde X}$ is therefore
  at least two-dimensional. Consequently, $R^1\pi_* \Omega^1_{\wtilde X}(\log E)
  \not = 0$ by the exact sequence~\eqref{eq:xs2}. This concludes the proof of
  Proposition~\ref{prop:fail1}.
\end{proof}

\subsubsection{Proof of Proposition~\ref*{prop:fail2}}
\label{ssec:popf2}

The proof of Proposition~\ref{prop:fail2} is similar to, but more involved than
the proof of Proposition~\ref{prop:fail1}. As in Section~\ref{ssec:popf1}, we
start giving a number of remarks and lemmas computing cohomology groups
relevant for our line of reasoning. Once these results are established, the
proof of Proposition~\ref{prop:fail2} follows on Page~\pageref{pf:prop:fail2}.

\begin{rem}[Hodge numbers of $Y$ and $\wtilde X$]\label{rem:ex2}
  It follows immediately from the construction of $Y$ and $\wtilde X$ as
  $\bP^1$-bundles that both spaces can be written as the disjoint union of
  affine subvarieties,
  $$
  Y = \bigcup^\bullet_i Y_i \quad \text{and} \quad \wtilde X =
  \bigcup^\bullet_j \wtilde X_j,
  $$
  such that the following holds.
  \begin{enumerate-p}
  \item\label{il:stoer} All $Y_i$, respectively $\wtilde X_j$ are algebraically
    isomorphic to affine spaces $\bC^{n_{i}}$, respectively $\bC^{n_j}$ where
    $n_{i}, n_j \in \bN$ are suitable numbers.
  \item\label{il:pollak} For all $i, j$, the irreducible components of
    $\overline{Y_i} \setminus Y_i$ and $\overline{\wtilde X_j} \setminus \wtilde
    X_j$ are smooth subvarieties of $Y$ and $\wtilde X$, respectively.
  \end{enumerate-p}
  Assertion~(\ref{rem:ex2}.\ref{il:stoer}) implies that the topological spaces
  $Y$ and $\wtilde X$ admit CW-structures without odd-dimensional cells. Using
  these CW-structures to compute (co)homology, cf.~\cite[Sect.~2.2]{Hatcher}, it
  follows that
  \begin{equation}\label{eq:thunfisch}
    H^k \bigl( Y,\, \bC \bigr) = H^k \bigl( \wtilde X,\, \bC \bigr) = 0 \quad
    \text{if $k$ is odd,}
  \end{equation}
  In fact, more is true. Property~(\ref{rem:ex2}.\ref{il:pollak}) implies that
  all topological cohomology groups $H^{2k}\bigl( Y,\, \bC \bigr)$ are generated
  by cohomology classes of smooth algebraic cycles, that is, by elements of
  $H^{k,k}\bigl(Y\bigr)$, cf.~\cite[Rem.~(iii) on p.~140]{Hatcher} and
  \cite[Prop.~11.20]{Voisin-Hodge1}. The same being true for $\wtilde X$, we
  find that
  \begin{equation}\label{eq:skrej}
    H^{k,l} \bigl( Y \bigr) = 0 \text{\,\, and\,\, }
    H^{k,l} \bigl( \wtilde X \bigr) = 0 \quad \text{if $k \not = l$.}
\end{equation}
\end{rem}

\begin{lem}\label{lem:A}
  In the setting of Example~\ref{ex:KANVisFalse}, we have
  \begin{enumerate-p}
  \item\label{il:matjes} $H^1 \bigl( Y,\, \Omega^2_Y \otimes \sO_Y(1) \bigr)
    \not = 0$, and
  \item\label{il:sardelle} $H^2 \bigl( Y,\, \Omega^2_Y \otimes \sO_Y(\nu)
    \bigr) = 0$ for all $\nu > 0$.
  \end{enumerate-p}
\end{lem}
\begin{proof}
  Let $\nu>0$ be any number.  We consider the sequence of relative differentials
  for the smooth morphism $\phi$,
  $$
  0 \to \phi^* \Omega^1_{\bP^2} \to \Omega^1_Y \to
  \underbrace{\Omega^1_{Y/\bP^2}}_{= \omega_{Y/\bP^2}} \to 0.
  $$
  Twisting the second exterior power of this sequence with $\sO_Y(\nu)$, one
  obtains
  \begin{equation}\label{eq:anchovi}
    0 \to \underbrace{\sO_Y(\nu) \otimes \phi^* \omega_{\bP^2}}_{ =: \sA_\nu}
    \to \sO_Y(\nu) \otimes \Omega^2_Y \to \underbrace{\phi^* \Omega^1_{\bP^2}
      \otimes \sO_Y(\nu) \otimes \omega_{Y/\bP^2}}_{=: \sB\nu} \to 0.
  \end{equation}
  Using Equation~\eqref{eq:KY} of Remark~\ref{rem:ex1} to expand the definition
  of $\omega_{Y/\bP^2}$, we obtain that
  \begin{equation}\label{eq:idBnu}
    \sB_\nu = \phi^* \big(\Omega^1_{\bP^2}\otimes \phi^* \omega_{\bP^2}^{-1}
    \bigr) \otimes \sO_Y(\nu-2).
  \end{equation}
  The following observations are crucial for cohomology computations.
  \begin{enumerate-p}
    \setcounter{enumi}{\value{equation}}
  \item\label{il:macarel} If $F$ is any fibre of $\phi$, then $F \cong \bP^1$,
    the restriction $\sA_\nu|_F$ is ample, and $H^1 \bigl( F,\, \sA_\nu|_F
    \bigr) = 0$ for all $\nu > 0$.  The Leray spectral sequence thus gives $H^i
    \bigl( Y,\, \sA_\nu \bigr) = H^i \bigl( \bP^2 ,\, \phi_* \sA_\nu \bigr)$ for
    all positive $\nu$ and $i$.

  \item\label{il:cod} Likewise, the Leray spectral sequence gives $H^i \bigl(
    Y,\, \sB_\nu \bigr) = H^i \bigl( \bP^2 ,\, \phi_* \sB_\nu \bigr)$ for all
    positive $i, \nu$.

  \item\label{il:tuna} The negativity of $\sB_1$ on $\phi$-fibres implies that
    $\phi_* \sB_1 = 0$. In particular, we obtain that $H^i\bigl( Y,\, \sB_1
    \bigr) = 0$ for all $i \geq 0$. 
  \end{enumerate-p}
  We compute cohomology groups first in case $\nu = 1$,
  \begin{align*}
    H^1\bigl( Y,\, \sA_1 \bigr) &= H^1\bigl( \bP^2,\, \phi_* \sA_1 \bigr) &&
    \text{by (\ref{lem:A}.\ref{il:macarel}}) \\
    &= H^1\bigl( \bP^2,\, \sT_{\bP^2} \otimes \omega_{\bP^2} \bigr) &&
    \text{since $Y = \bP_{\bP^2}(\sT_{\bP^2})$, so $\phi_* \sO_Y(1) =  \sT_{\bP^2}$}\\
    &\cong H^1\bigl( \bP^2,\, \Omega^1_{\bP^2} \bigr)^* \cong \bC && \text{Serre duality.} \\
    H^2\bigl( Y,\, \sA_1 \bigr) &\cong H^0\bigl( \bP^2,\, \Omega^1_{\bP^2} \bigr)^* = 0.  && \text{Analogous computation.}
  \end{align*}
  The long exact cohomology sequence for \eqref{eq:anchovi} in case $\nu = 1$,
  $$
    \cdots \to \underbrace{H^0\bigl( Y,\, \sB_1 \bigr)}_{= 0 \text{ by
        (\ref{lem:A}.\ref{il:tuna})}} \to \underbrace{H^1\bigl( Y,\, \sA_1
      \bigr)}_{\cong \bC} \to H^1 \bigl( Y,\, \sO_Y(1) \otimes \Omega^2_Y \bigr)
    \to \cdots, 
    $$
  then shows Claim~(\ref{lem:A}.\ref{il:matjes}). Another excerpt of the same sequence,
  $$
  \cdots \to \underbrace{H^2\bigl( Y,\, \sA_1 \bigr)}_{= 0} \to H^2\bigl( Y,\,
  \sO_Y(1) \otimes \Omega^2_Y \bigr) \to \underbrace{H^2 \bigl( Y,\, \sB_1
    \bigr)}_{= 0 \text{ by (\ref{lem:A}.\ref{il:tuna})}} \to \cdots,
  $$
  already shows Claim~(\ref{lem:A}.\ref{il:sardelle}) in case $\nu = 1$. 
  
  It remains to show Claim~(\ref{lem:A}.\ref{il:sardelle}) for $\nu > 1$. We
  identify the following cohomology groups.
  \begin{align*}
    H^2 \bigl( Y,\, \sA_\nu \bigr) &= H^2 \bigl( \bP^2,\, \phi_* \sA_{\nu}
    \bigr) && \text{by (\ref{lem:A}.\ref{il:macarel})} \\
    & = H^2 \bigl( \bP^2,\, \Sym^\nu \sT_{\bP^2} \otimes \omega_{\bP^2} \bigr)
    && \text{since $Y = \bP_{\bP^2}(\sT_{\bP^2})$} \\
    & \cong H^0 \bigl( \bP^2,\, \Sym^\nu \Omega^1_{\bP^2} \bigr)^* && \text{Serre duality} \\
    & = 0 && \text{since $\Omega^1_{\bP^2}$ is anti-ample} \\
    H^2 \bigl( Y,\, \sB_\nu \bigr) &= H^2 \bigl( \bP^2,\, \phi_* \sB_{\nu}
    \bigr) && \text{by (\ref{lem:A}.\ref{il:cod})} \\
    & = H^2 \bigl( \bP^2,\, \Omega^1_{\bP^2} \otimes \omega^{-1}_{\bP^2}
    \otimes \Sym^{\nu-2} \sT_{\bP^2} \bigr) && \text{by \eqref{eq:idBnu}} \\
    & \cong H^0 \bigl( \bP^2,\, \underbrace{\sT_{\bP^2} \otimes \Sym^{\nu-2} \Omega^1_{\bP^2} \otimes \sO_{\bP^2}(-6)}_{=: \sC_\nu} \bigr)^* && \text{by Serre duality} \\
    & = 0
  \end{align*}
  The last equality holds because the restriction $\sC_{\nu}$ to any line $l $ is
  obviously negative. The long exact cohomology sequence for \eqref{eq:anchovi}
  will then immediately give Claim~(\ref{lem:A}.\ref{il:sardelle}) for $\nu >
  1$, as required.
\end{proof}

\begin{cor}\label{cor:B}
  In the setting of Example~\ref{ex:KANVisFalse}, we have $H^1 \bigl( E,\,
  \Omega^3_{\wtilde X}\bigl|_E \bigr) \not = 0$.
\end{cor}
\begin{proof}
  The divisor $E \cong Y$ being a section of $\psi: \wtilde X \to Y$, we see
  that the normal bundle sequence of $E$ splits, that is, $\Omega^1_{\wtilde
    X}\bigl|_E = \Omega^1_E \oplus N^*_{E/\wtilde X} \cong \Omega^1_Y \oplus
  \sO_Y(1)$. Taking cohomology of the third wedge-power, we obtain that
  $$
  H^1 \bigl( E,\, \Omega^3_{\wtilde X}\bigl|_E \bigr) \cong H^1
    \bigl( Y,\, \Omega^3_Y \bigr) \oplus \underbrace{H^1\bigl( Y,\, \Omega^2_Y \otimes
    \sO_Y(1) \bigr)}_{\mathclap{\not = 0 \text{ by Assertion~(\ref{lem:A}.\ref{il:matjes})
      of Lemma~\ref{lem:A}}}}.
  $$
\end{proof}

\begin{cor}\label{cor:C}
  In the setting of Example~\ref{ex:KANVisFalse}, let $\what E$ be the formal
  completion of the $\pi$-exceptional divisor $E$ in $\wtilde X$. Then $H^1
  \bigl( \what E,\, \Omega^3_{\wtilde X}\bigl|_{\what E} \bigr) \not = 0$.
\end{cor}
\begin{proof}
  We follow the standard approach and compute cohomology on $\what E$ as an
  inverse limit of cohomology on higher-order infinitesimal neighborhoods of $E$.
  Denoting by $E_n \subset \wtilde X$ the subscheme defined by $\sJ_E^{n+1}$,
  recall from \cite[Prop.~4.1]{Ha68} that the cohomology on the formal
  completion is computed as
  $$
  H^1 \bigl( \what E,\, \Omega^3_{\wtilde X}\bigl|_{\what E} \bigr) =
  \lim_{\leftarrow} H^1\bigl( E_n,\, \Omega^3_{\wtilde X}\bigl|_{E_n} \bigr),
  $$
  where the limit is taken over the inverse system given by restriction maps
  $$
  \cdots \xrightarrow{r_3} H^1\bigl( E_3,\, \Omega^3_{\wtilde X}\bigl|_{E_3}
  \bigr) \xrightarrow{r_2} H^1\bigl( E_2,\, \Omega^3_{\wtilde X}\bigl|_{E_2}
  \bigr) \xrightarrow{r_1} \underbrace{H^1\bigl( E,\, \Omega^3_{\wtilde
      X}\bigl|_E \bigr)}_{\not = 0 \text{ by Corollary~\ref{cor:B}}}.
  $$
  Corollary~\ref{cor:C} will thus follow once we show that all restriction
  morphisms $r_i$ are surjective. To this end, fix a number $n>1$ and consider
  the sequence of coherent $\sO_{\wtilde X}$-modules,
  $$
  0 \to \underbrace{\factor \sJ^n_E.\sJ^{n+1}_E.}_{\cong ( N_{E/\wtilde
      X}^*)^{\otimes n}} \to \underbrace{\factor \sO_{\wtilde
      X}.\sJ^{n+1}_E.}_{\cong \sO_{E_n}} \to \underbrace{\factor \sO_{\wtilde
      X}.\sJ^n_E.}_{\cong \sO_{E_{n-1}}} \to 0.
  $$
  Tensoring over $\sO_{\wtilde X}$ with $\Omega^3_{\wtilde X}$, the associated
  long exact cohomology sequence reads
  $$
  H^1\bigl( E_n,\, \Omega^3_{\wtilde X}\bigl|_{E_n} \bigr) \xrightarrow{r_{n-1}}
  H^1\bigl( E_{n-1},\, \Omega^3_{\wtilde X}\bigl|_{E_{n-1}} \bigr) \to
  \underbrace{H^2\bigl( E,\, \Omega^3_{\wtilde X}\bigl|_E \otimes \sO_E(n)
    \bigr)}_{\mathclap{= 0 \text{ by Assertion~\ref{lem:A}.\ref{il:sardelle} of
        Lemma~\ref{lem:A}}}},
  $$
  giving the surjectivity we needed to show.
\end{proof}

\begin{lem}\label{lem:D}
  In the setting of Example~\ref{ex:KANVisFalse}, we have $H^1 \bigl( \wtilde
  X,\, \Omega^3_{\wtilde X} \otimes \sL \bigr) = 0$.
\end{lem}
\begin{proof}
  Let $\Sigma \in | \pi^* \sL|$ be a general element. The divisor $\Sigma
  \subset \wtilde X$ is then a section for the map $\psi: \wtilde X \to Y$. In
  particular, $\Sigma \cong Y$ and $N_{\Sigma/\wtilde X} \cong \sL|_{\Sigma}
  \cong \sO_Y(1)$. Twisting the ideal sheaf sequence for $\Sigma \in \wtilde X$
  with $\Omega^3_{\wtilde X} \otimes \sL$, we obtain
  \begin{equation}\label{eq:dorade}
    0 \to \Omega^3_{\wtilde X} \to \Omega^3_{\wtilde X} \otimes \sL \to \bigl(
    \Omega^3_{\wtilde X} \otimes \sL\bigr)\bigl|_{\Sigma} \to 0.
  \end{equation}
  Using the long exact cohomology sequence of~\eqref{eq:dorade},
  Lemma~\ref{lem:D} follows once we show that
  \begin{align}
    \label{van:whale} & H^1\bigl(\wtilde X,\, \Omega^3_{\wtilde X}\bigr) =
    H^{1,3}\bigl( \wtilde X \bigr) = 0
    &&\text{and} \\
    \label{van:shark} & H^1\bigl(\Sigma,\, \Omega^3_{\wtilde X} \otimes
    \sL|_\Sigma \bigr) = 0.
  \end{align}
  Vanishing~\eqref{van:whale} has already been shown in Remark~\ref{rem:ex2}
  above. For~\eqref{van:shark}, use the mapping $\psi$ to write
  $\Omega^1_{\wtilde X}\bigl|_\Sigma \cong \Omega^1_\Sigma \oplus N_{\Sigma/\wtilde
    X}^*$ so that
  $$
  \bigl( \Omega^3_{\wtilde X} \otimes \sL\bigr) \bigl|_\Sigma \cong \bigl(
  \omega_\Sigma \otimes \sL|_\Sigma\bigr) \, \oplus \, \bigl(
  \Omega^2_{\Sigma} \otimes N_{\Sigma/\wtilde X}^* \otimes \sL|_\Sigma \bigr)
  \cong \bigl( \omega_\Sigma \otimes \sL|_\Sigma \bigr) \, \oplus \,
  \Omega^2_{\Sigma}
  $$
  and
  $$
  H^1\bigl( \Sigma,\, (\Omega^3_{\wtilde X} \otimes \sL)|_\Sigma \bigr) \cong
  \underbrace{H^1\bigl( Y,\, \omega_Y \otimes \sO_Y(1) \bigr)}_{=0\text{ by
      Kodaira vanishing}} \quad \oplus \quad \underbrace{H^1\bigl( Y,\,
    \Omega^2_Y \bigr)}_{\mathclap{= H^{1,2}(Y) = 0 \text{ by
        Remark~\ref{rem:ex2}}}}.
  $$
  This finishes the proof of Lemma~\ref{lem:D}.
\end{proof}

\begin{proof}[Proof of Proposition~\ref*{prop:fail2}]\label{pf:prop:fail2}
  Consider the following excerpt of the 5-term exact sequence associated with
  the Leray spectral sequence for the sheaf $\Omega^3_{\wtilde X} \otimes \pi^*
  \sL$,
  $$
  \cdots \to %
  \underbrace{H^1 \bigl( \wtilde X,\, \Omega^3_{\wtilde X} \otimes \pi^* \sL
    \bigr)}_{= 0 \text{ by Lemma~\ref{lem:D}}} \to %
  \underbrace{H^0 \bigl( X,\, R^1\pi_* \Omega^3_{\wtilde X} \otimes \sL
    \bigr)}_{=E^{0,1}_2} \to %
  H^2 \bigl( X,\, \underbrace{\pi_* \Omega^3_{\wtilde X} \otimes \sL}_{\mathclap{=
    \Omega^{[3]}_X \otimes \sL \text{ by the ext.~theorem}}} \bigr) \to %
  \cdots
  $$
  To show that $H^2 \bigl( X,\, \Omega^{[3]}_X \otimes \sL \bigr) \not = 0$, is
  suffices to show that $E^{0,1}_2 \not = 0$. Since the higher direct image
  sheaf vanishes outside of the singular point $x \in X$, this will follow from
  $R^1\pi_* \Omega^3_{\wtilde X} \not = 0$. Non-vanishing of $R^1\pi_*
  \Omega^3_{\wtilde X}$ follows from Corollary~\ref{cor:C} and from the theorem
  on formal functions, in the form of \cite[Chapt.~III,
  Cor.~4.1.7]{EGA3-1}. This finishes the proof of Proposition~\ref{prop:fail2}.
\end{proof}

\section{Closedness of reflexive forms, Poincaré's lemma}\label{sect:Poincare}

\subsection{Closedness of reflexive forms}

It is a standard result of Hodge theory that logarithmic differential forms on
projective snc pairs are closed. Here, we show that similar results also hold
for reflexive differentials, in the log canonical setting. The following
notation is fundamental in the discussion.

\begin{defn}[Closedness for reflexive forms]\label{def:closeness}
  Let $X$ be a normal algebraic variety (or complex space), $E$ a reduced effective divisor on $X$
  and $\sigma \in H^0\bigl(X, \Omega^{[p]}_X(\log E) \bigr)$ a reflexive
  $p$-form. We call $\sigma$ \emph{closed} if $d(\sigma \vert_{X_{\reg}
    \setminus E}) = 0$.
\end{defn}

While a Hodge theory for reflexive differentials is still missing, we prove the
following generalisation of the particular Hodge-theoretic result mentioned
above.

\begin{thm}[Closedness of global logarithmic forms]\label{thm:closedness}
  Let $X$ be a projective variety, and let $D$ be an effective $\bQ$-divisor such that the
  pair $(X, D )$ is log canonical. Then, any reflexive logarithmic $p$-form
  $\sigma \in H^0\bigl(X,\, \Omega^{[p]}_X(\log \lfloor D \rfloor) \bigr)$ is
  closed.
\end{thm}
\begin{proof}
  Let $\pi: \wtilde X \to X$ be a strong log-resolution of the pair $(X, D)$ and
  set
  $$
  \wtilde D := \text{the largest reduced divisor contained in } \supp
  \pi^{-1}(\nklt (X,D)).
  $$
  Now, if $\sigma \in H^0\bigl(X,\, \Omega^{[p]}_X(\log \lfloor D \rfloor)
  \bigr)$ is any reflexive $p$-form on $X$, it follows from the Extension
  Theorem, \cite[Thm.~1.5]{GKKP11}, that the pull back of $\sigma$ via $\pi$
  extends to an element $\wtilde \sigma \in H^0\bigl(\wtilde X,\,
  \Omega^p_{\wtilde X} (\log \wtilde D) \bigr)$. However, global logarithmic
  forms on the projective snc pair $(\wtilde X, \wtilde D)$ are closed
  by~\cite[(3.2.14)]{DeligneHodgeII}. In particular,
  $$
  d\bigl(\wtilde{\sigma}|_{\wtilde X \setminus \what D} \bigr)= 0, \quad
  \text{where } \what D := (\text{$\pi$-exceptional set}) \cup \supp \wtilde D.
  $$
  Since $\pi$ identifies $\wtilde X \setminus \what D$ with $X_{\reg} \setminus
  \supp \lfloor D \rfloor$, this shows that $\sigma$ is closed in the sense of
  Definition~\ref{def:closeness}.
\end{proof}

\begin{rem}\label{rem:holclosed}
  Using the Holomorphic Extension Theorem~\ref{thm:holomorphicextension}, an
  analogous result can be shown to hold for (logarithmic) differential forms on
  log canonical locally algebraic compact complex spaces in class
  $\mathscr{C}$. In particular, this applies to Moishezon spaces with log
  canonical singularities.
\end{rem}

\subsection{The Poincaré Lemma in the reflexive setting}

Poincaré's lemma is one of the cornerstones in the theory of differential forms
on complex manifolds. For (analytically) klt spaces, the lemma still holds for
reflexive one-forms. It is currently unclear to what extent a Poincaré lemma can
be expected to hold for reflexive forms of higher degree.

\begin{thm}[Poincaré Lemma for reflexive one-forms]\label{thm:poincare-1}
  Let $X$ be a normal complex space and $D$ an effective $\bQ$-divisor on $X$
  such that $(X,D)$ is analytically klt and locally algebraic. Let $\sigma \in
  H^0\bigl( X,\, \Omega^{[1]}_X \bigr)$ be a closed holomorphic reflexive
  one-form on $X$. Then there exists a covering of $X$ by subsets
  $(U_\alpha)_{\alpha \in A}$ that are open in the euclidean topology, and
  holomorphic functions $f_\alpha \in \sO_{X}(U_\alpha)$ such that
  $\sigma|_{U_{\alpha,\reg}} = df_\alpha|_{U_{\alpha,\reg}}$.
\end{thm}

\begin{subrem}
  For isolated rational singularities, slightly more general results have been
  obtained in \cite[Prop.~2.5]{CampanaFlennerContactSings}.
\end{subrem}

\begin{subrem}
  If the pair $(X,D)$ is only assumed to be log-canonical, the Poincaré lemma
  fails to be true in general. As an example, consider the affine cone $X$ over
  a smooth plane cubic curve $C \subset \mathbb{P}^2$, together with its minimal
  resolution $\widetilde X \to X$. The variety $\widetilde X$ is the total space
  of the line bundle $\mathscr{O}_C(-1)$, hence it fibres over the exceptional
  set $E \subset \widetilde X$. Pulling-back a nowhere vanishing global regular
  one-form from $E\cong C$ to $\widetilde X$ yields a closed regular one-form
  which does not have a primitive in any euclidean neighbourhood of $E$ in
  $\widetilde X$.
\end{subrem}

\begin{proof}[Proof of Theorem~\ref{thm:poincare-1}]
  Let $x \in X$ be any point. We aim to construct an open neighborhood $U =
  U(x)$ and a function $f \in \sO_X(U)$ satisfying the requirements of the
  theorem. To this end, consider a log resolution $\pi : \wtilde X \to X$ such
  that both the $\pi$-exceptional set $E$ and the fiber $F := \pi^{-1}(x)$ are
  divisors with simple normal crossing support. Let $F = \cup_i F_i$ be the
  decomposition into irreducible components.

  In this setting, the holomorphic Extension
  Theorem~\ref{thm:holomorphicextension} guarantees that there exists a
  differential form $\wtilde \sigma \in H^0\bigl( \wtilde X,\, \Omega^1_{\wtilde
    X} \bigr)$ which agrees over the smooth part $X_{\reg}$ with the pull-back
  of $\sigma$. Using the classical Poincaré Lemma and elementary topology, we
  find a finite number of sets $(V_\beta)_{\beta \in B} \subset \wtilde X$, open
  in the euclidean topology, which cover $F$ and satisfy the following
  additional properties.
  \begin{enumerate-p}
    \setcounter{enumi}{\value{equation}}
  \item\label{il:x1} For each index $\beta \in B$, the intersection $V_\beta
    \cap F$ is not empty and connected.
  \item\label{il:x2} Given indices $\beta_1, \beta_2 \in B$ with $V_{\beta_1}
    \cap V_{\beta_2} \not = \emptyset$, then $V_{\beta_1} \cap V_{\beta_2}$ is
    connected and $V_{\beta_1} \cap V_{\beta_2} \cap F \not = \emptyset$.
  \item\label{il:x3} For each index $\beta \in B$, there exists a holomorphic
    function $g_\beta \in \sO_{\wtilde X}(V_\beta)$ such that $\wtilde
    \sigma|_{V_\beta} = dg_\beta$.
  \end{enumerate-p}
  To continue, we recall a result of Namikawa
  \cite[Lem.~1.2]{NamikawaDeformationTheory}, which asserts in our setup that
  $H^0 \bigl( F, \Omega^1_F / \tor \bigr) = 0$. In particular, if $i$ is any
  index and $\iota_i : F_i \to \wtilde X$ the inclusion map, then
  \begin{equation}\label{eq:diota}
    d\iota_i (\wtilde \sigma) = 0 \in H^0 \bigl( F_i, \, \Omega^1_{F_i} \bigr)    
  \end{equation}
  Since $d(\iota_i|_{V_\beta\cap F_i }) (dg_\beta) = d(g_\beta|_{V_\beta\cap F_i
  }) = 0$, Equation~\eqref{eq:diota} implies that the functions $g_\beta$ are
  locally constant along $V_\beta \cap F$. Using
  (\ref{thm:poincare-1}.\ref{il:x1}), we can thus assume the following.
  \begin{enumerate-p}
    \setcounter{enumi}{\value{equation}}
  \item\label{il:x4} For each index $\beta \in B$, the function $g_\beta$ of
    (\ref{thm:poincare-1}.\ref{il:x3}) vanishes along $F$, that is, $g_\beta \in
    \sJ_{F}(V_\beta) \subset \sO_{\wtilde X}(V_\beta)$.
  \end{enumerate-p}
  Now, given any two indices $\beta_1, \beta_2 \in B$ with $V_{\beta_1} \cap
  V_{\beta_2} \not = \emptyset$, it follows from
  (\ref{thm:poincare-1}.\ref{il:x3}) that the difference
  $g_{\beta_1}|_{V_{\beta_1} \cap V_{\beta_2}} - g_{\beta_2}|_{V_{\beta_1} \cap
    V_{\beta_2}}$ is locally constant, hence constant by
  (\ref{thm:poincare-1}.\ref{il:x2}). Using (\ref{thm:poincare-1}.\ref{il:x4}),
  we see that this difference is actually zero along the non-empty set
  $V_{\beta_1} \cap V_{\beta_2} \cap F$. In summary, the functions $g_\beta$
  glue to give a globally defined holomorphic function $g \in \sO_{\wtilde
    X}(\cup V_\beta)$. Since $\pi$ is proper, we find a neighborhood $U = U(x)$,
  open in the euclidean topology, such that $\pi^{-1}(U) \subset \cup V_\beta$.
  The existence of a holomorphic function $f \in \sO_{X}(U)$ satisfying $g = f
  \circ \pi$ is then immediate.
\end{proof}

\begin{rem}
  In the setting of the proof, Namikawa's vanishing $H^0 \bigl( F, \Omega^1_F /
  \tor \bigr) = 0$ can also be shown by elementary methods, using that $F$ is
  rationally chain connected.
\end{rem}

\subsection{Representing reflexive forms by Kähler differentials}

Given a normal variety or a normal complex space $X$ , there exists a natural
morphism $b : \Omega^p_X \to \Omega_X^{[p]}$ from the sheaf of Kähler
differentials to its double dual. In general, the morphism $b$ is neither
injective nor surjective, even if $X$ is assumed to have the mildest possible
singularities considered in the Minimal Model Program. A corresponding series of
examples is worked out in detail in \cite{GrebRollenske}.

The kernel of $b$ is exactly the subsheaf of torsion elements in $\Omega^1_X$,
so that there is an exact sequence
\begin{equation}\label{eq:xxB1}
  \xymatrix{ %
    0 \ar[r] &  \tor \Omega^1_X \ar[r]^a & \Omega^1_X \ar[r]^b & \Omega^{[1]}_X. \\
  }
\end{equation}

\begin{thm}[Representation of closed forms by Kähler differentials]\label{thm:xx}
  Let $X$ be an irreducible normal complex space and $D$ an effective $\mathbb
  Q-$divisor such that $(X,D)$ is analytically klt and locally algebraic.
  \begin{enumerate-p}
    \setcounter{enumi}{\value{equation}}
  \item \label{il:Poincare} If $\sigma \in H^0 \bigl( X,\, \Omega^{[1]}_X
    \bigr)$ is any closed reflexive one-form, then there exists a canonically
    defined Kähler form $\sigma_K \in H^0 \bigl( X,\, \Omega^{1}_X \bigr)$ such
    that $\sigma = b(\sigma_K)$.
  \item \label{il:projectivePoincare} If $X$ is a projective algebraic variety
    with at worst klt singularities, then the sequence
    $$
    \xymatrix{ %
      & \quad 0 \ar[r] & H^0\bigl( X,\, \tor \Omega^1_X \bigr) \ar[r]^{\overline{a}} & H^0\bigl( X,\, \Omega^1_X
      \bigr) \ar[r]^{\overline{b}} & H^0\bigl( X,\, \Omega^{[1]}_X \bigr) \ar[r] & 0
    }
    $$
    is exact and canonically split.
  \end{enumerate-p}
\end{thm}
\begin{proof}
  Assertion~(\ref{thm:xx}.\ref{il:Poincare}) follows immediately from the
  Poincaré-Lemma~\ref{thm:poincare-1} above. In order to prove Assertion
  (\ref{thm:xx}.\ref{il:projectivePoincare}), let $X$ be a projective algebraic
  variety with at worst klt singularities. By Theorem~\ref{thm:closedness} and
  Assertion~(\ref{thm:xx}.\ref{il:Poincare}), the torsion sequence for analytic
  Kähler differentials yields the following exact and canonically split
  sequence,
  \begin{equation}\label{eq:holcanonicallysplit}
    \begin{xymatrix}{ %
        H^0\bigl( X^{an},\, \tor \Omega^1_{X^{an}} \bigr) \ar@{^(->}[r] & H^0\bigl( X^{an},\, \Omega^1_{X^{an}}
        \bigr) \ar@{->>}[r] & H^0\bigl( X^{an},\, \Omega^{[1]}_{X^{an}} \bigr).
      }
    \end{xymatrix}
  \end{equation}
  To compare the torsion sequences for analytic and algebraic Kähler
  differentials we consider the following commutative diagram, both rows of
  which are exact.
  $$
  \begin{xymatrix}{
      0 \ar[r]& \mathrm{tor}\,\Omega_{X^{an}}^1 \ar[r] & \Omega_{X^{an}}^1 \ar[r]& \Omega^{[1]}_{X^{an}}\ar[r] & 0 \\
      0 \ar[r]& \bigl(\mathrm{tor}\, \Omega_X^1\bigr)^{an} \ar[r]\ar[u]_\alpha & \bigl(\Omega_X^1\bigr)^{an}\ar[r]\ar[u]_\beta & \bigl(\Omega_X^{[1]} \bigr)^{an}\ar[r]\ar[u]_\gamma & 0
    }
  \end{xymatrix}
  $$
  By the functoriality properties of the sheaf of Kähler differentials $\beta$
  is an isomorphism. Additionally, it follows from equation
  \eqref{eq:homandhomagain} that $\gamma$ is isomorphic. Consequently, $\alpha$
  is likewise an isomorphism.

  It hence follows from GAGA \cite[Thm.~1]{GAGA} that in the following
  commutative diagram the vertical maps are isomorphic.
  $$
  \begin{xymatrix}{
      0 \ar[r]& H^0\bigl(X^{an},\, \mathrm{tor}\,\Omega_{X^{an}}^1\bigr) \ar[r] & H^0\bigl(X^{an},\,\Omega_{X^{an}}^1\bigr) \ar[r]& H^0\bigl(X^{an},\, \Omega^{[1]}_{X^{an}} \bigr)\ar[r] & 0 \\
      0 \ar[r]& H^0\bigl(X,\, \mathrm{tor}\, \Omega_X^1\bigr)
      \ar[r]^{\overline{a}}\ar[u]_{\overline{\alpha}} & H^0\bigl(X,\,
      \Omega_X^1\bigr)\ar[r]^{\overline{b}}\ar[u]_{\overline{\beta}} &
      H^0\bigl(X,\, \Omega_X^{[1]} \bigr)\ar[u]_{\overline{\gamma}} \ar[r] & 0 }
  \end{xymatrix}
  $$
  The upper row coincides with \eqref{eq:holcanonicallysplit} and is therefore
  exact and canonically split. Consequently, the same holds for the lower row,
  which coincides with the sequence of
  (\ref{thm:xx}.\ref{il:projectivePoincare}). This concludes the proof.
\end{proof}

\begin{rem}
  For projective varieties, exactness and splitting of
  Sequence~(\ref{thm:xx}.\ref{il:projectivePoincare}) can also be concluded from
  the fact that the Albanese map of any desingularisation factors via $X$, and
  that any 1-form is a pull-back from the Albanese torus.
\end{rem}

\begin{question}
  Are there similar results for reflexive $p$-forms, for $p > 1$? If in the
  setup of Theorem~\ref{thm:xx} both the pair $(X,D)$ and the form $\sigma$ are
  algebraic, is $\sigma_K$ also algebraic? ---The last question has a positive
  answer in the case of isolated singularities.
\end{question}

\appendix

\section{Stability notions on singular spaces}
\label{sec:snsp}

The proof of Theorem~\ref{thm:rc} uses the notion of semistable sheaves on
singular spaces, where semistability is defined with respect to a fixed movable curve
class. As we are not aware of any reference that discusses these matters in
detail, we chose to include a short and self-contained introduction here. We
feel that these results might be of independent interest.

\subsection{Semistability with respect to a movable class}
\label{ssec:swrtm}

On a polarized complex manifold, it is well-understood that the tensor product
of any two semistable locally free sheaves is again semistable. We will show in
this appendix that the reflexive tensor product of two semistable sheaves on a
singular, $\mathbb Q$-factorial space is again semistable, even if the
polarization is only given by a movable curve class, see
Proposition~\ref{prop:SSreflTensor} below. To start, we recall the relevant
definition of semistability with respect to movable classes.

\begin{defn}[Semistability with respect to a movable class]\label{def:semistability}
  Let $X$ be a normal $\mathbb Q$-factorial projective variety and $\alpha \in
  N_1(X)_{\mathbb Q}$ a numerical curve class\footnote{See
    \cite[Sect.~II.4]{K96} for the definitions of the spaces $N_1(X)_{\bQ}$ and
    $N^1(X)_{\bQ}$, and for all relevant facts used here.}. Assume that $\alpha$
  is \emph{movable}, that is, that $\alpha$ intersects any effective Cartier
  divisor non-negatively.

  If $\sF$ is any coherent sheaf of $\sO_X$-modules that is torsion free in
  codimension one, the determinant $\det \sF := (\wedge^{\rank \sF} \sF)^{**}$
  is a $\mathbb Q$-Cartier Weil divisorial sheaf and therefore defines a class
  $[\det \sF] \in N^1(X)_{\bQ}$. Using the non-degenerate bilinear pairing
  \begin{equation}\label{eq:blp}
    N_1(X)_{\bQ} \times N^1(X)_{\bQ} \to \bQ,
  \end{equation}
  one defines the \emph{slope of $\sF$ with respect to $\alpha$} as the rational
  number
  \begin{equation}\label{eq:d1}
    \mu_{\alpha}(\sF) := \frac{ [ \det \sF ] . \alpha}{\rank(\sF)}.    
  \end{equation}
  Recalling that subsheaves of $\sF$ are again torsion free in codimension one,
  we set
  \begin{equation}\label{eq:d2}
    \mu_{\alpha}^{\max} (\sF) := \sup \left\{ \mu_\alpha(\sG) \,|\, \text{$\sG
        \subseteq \sF$ a coherent subsheaf} \right\}
  \end{equation}
  We say that $\sF$ is \emph{$\alpha$-semistable} if $\mu_{\alpha}^{\max}(\sF) =
  \mu_{\alpha}(\sF)$.
\end{defn}

The following proposition asserts that $\mu_{\alpha}^{\max}$ is never
infinite. The supremum used in its definition is actually a maximum.

\begin{prop}\label{prop:mumax}
  In the setting of Definition~\ref{def:semistability}, there exists a coherent
  subsheaf $\sG \subseteq \sF$ such that $\mu_{\alpha}^{\max} (\sF) =
  \mu_\alpha(\sG)$. In particular, $\mu_{\alpha}^{\max} (\sF)$ is a rational
  number.
\end{prop}

The remainder of the present Section~\ref{ssec:swrtm} is devoted to the proof of
Proposition~\ref{prop:mumax}. We have subdivided the proof into several
relatively independent steps, given in
Sections~\ref{ssec:popmum0}--\ref{ssec:popmum5}, respectively. The subsequent
Sections~\ref{ssec:AB}--\ref{pf:prop14} establish semistability of tensor
products, using invariance properties of slopes to reduce the problem to known
cases.

\subsubsection{Proof of Proposition~\ref*{prop:mumax}, invariance of $\mu_\alpha^{\max}$ under removal of torsion}
\label{ssec:popmum0}

We maintain notation and assumptions of Proposition~\ref{prop:mumax} throughout
the proof. The following observation will help to simplify the setting.

\begin{obs}\label{obs:remtor}
  Consider the natural quotient map 
  $$
  p : \sF \to \check \sF := \sF/\tor.
  $$
  If $\sG \subseteq \sF$ is any coherent subsheaf, it is clear that $\sG$ and
  $\check \sG := p(\sG)$ differ only along a set of codimension two, if at
  all. It follows that $\mu_\alpha(\check \sG) = \mu_\alpha(\sG)$. Likewise,
  given any coherent subsheaf $\check \sG \subseteq \check \sF$, set $\sG :=
  p^{-1}(\check \sG)$. Again, $\sG$ and $\check \sG$ differ only along a small
  set, and $\mu_\alpha(\check \sG) = \mu_\alpha(\sG)$.
\end{obs}

In summary, Observation~\ref{obs:remtor} shows that $\mu_\alpha^{\max}(\sF) =
\mu_{\alpha}^{\max} (\check \sF)$. Better still, there exists a sheaf $\sG
\subseteq \sF$ such that $\mu_\alpha(\sG) = \mu_{\alpha}^{\max} (\sF)$ if and
only if there exists a sheaf $\check \sG \subseteq \check \sF$ such that
$\mu_\alpha(\check \sG) = \mu_{\alpha}^{\max} (\check \sF)$. We will therefore
make the following assumption for the remainder of the proof.

\begin{awlog}\label{awlog:tors}
  The sheaf $\sF$ is torsion-free.
\end{awlog}

\subsubsection{Proof of Proposition~\ref*{prop:mumax}, bounding $\mu_{\alpha}^{\max} (\sF)$}
\label{ssec:popmum1}

As a first step towards a proof of Proposition~\ref{prop:mumax}, the following
Lemma~\ref{lem:mumax} asserts that the slopes $\mu_\alpha(\sG)$ of subsheaves
$\sG \subseteq \sF$ are uniformly bounded from above. Its proof uses only fairly
standard arguments, see for instance \cite[p.~62]{MP97}.

\begin{lem}\label{lem:mumax}
  Setting as above. Then $\mu_{\alpha}^{\max} (\sF) < \infty$.
\end{lem}
\begin{proof}
  Fix a very ample Cartier divisor $H$. Since $\sF$ is torsion-free, we find
  positive numbers $m$ and $N$ and an inclusion $\sF \subseteq \sO_X(mH)^{\oplus
    N}$.  Since $\mu_{\alpha}^{\max} (\sF) \leq \mu_{\alpha}^{\max}
  \bigl(\sO_X(mH)^{\oplus N}\bigr)$, we may assume without loss of generality
  for the remainder of the proof that $\sF = \sO_X(mH)^{\oplus N}$.

  We claim that the numerical class of any coherent subsheaf $\sG \subseteq \sF$
  satisfies the inequality
  \begin{equation}\label{eq:Srg}
    [\det \sG] \cdot \alpha \leq Nm[H] \cdot \alpha.    
  \end{equation}
  In fact, given any $\sG \subseteq \sF = \sO_X(mH)^{\oplus N}$, consider the
  following commutative diagram with exact rows,
  \begin{equation}\label{eq:Srk}
    \xymatrix{ %
      0 \ar[r] & \sG \ar[r] \ar@{^(->}[d] & \sO_X(mH)^{\oplus N} \ar[r] \ar@{=}[d] & \sQ \ar[r] \ar@{->>}[d] & 0 \ar@<.8cm>@{}[d]^{\txt{\small where\qquad\quad\ \\ \small $\sQ := \sF/\sG$\phantom{,} \\ \small $\sQ' := \sQ / \tor$}} \\
      0 \ar[r] & \sG' \ar[r]  & \sO_X(mH)^{\oplus N} \ar[r] & \sQ' \ar[r]  & 0.
    }
  \end{equation}
  The sheaf $\sG'$ is called ``saturation of $\sG$ in $\sO_X(mH)^{\oplus
    N}$''. The sheaf $\sG'$ and its subsheaf $\sG$ have the same
  rank. Consequently, there exists an effective divisor $D$ and an equality of
  numerical classes
  $$
  [\det \sG'] = [\det \sG] + [D].
  $$
  Since the curve class $\alpha$ is movable, this implies that $[\det \sG] \cdot
  \alpha \leq [\det \sG'] \cdot \alpha$.

  Since $H$ is assumed to be very ample, the sheaf $(\det \sQ')$ has a
  non-trivial section, and is therefore represented by an effective divisor. It
  follows that $[\det \sQ'] \cdot \alpha \geq 0$. Using that $\sQ'$ is
  torsion-free, we obtain an equality of numerical divisor classes,
  $$
  [\det \sG'] = Nm [H] - [\det \sQ'] \in N^1(X)_{\bQ},
  $$
  which yields the desired inequality of intersection numbers,
  $$ 
  [\det \sG] \cdot \alpha \leq  [\det \sG'] \cdot \alpha \leq Nm[H] \cdot \alpha.
  $$
  This shows Inequality~\eqref{eq:Srg} and finishes the proof of
  Lemma~\ref{lem:mumax}.
\end{proof}

\subsubsection{Proof of Proposition~\ref*{prop:mumax}, setup for Reducio Ad Absurdum}

It remains to show that $\mu_{\alpha}^{\max}(\sF) = \mu_{\alpha}(\sG)$, for a
suitable sheaf $\sG \subseteq \sF$.  We argue by contradiction and assume that
this is not the case.

\begin{assumption}\label{awlog:cont}
  There is no subsheaf $\sG \subseteq \sF$ such that $\mu_{\alpha}(\sG) =
  \mu_{\alpha}^{\max}(\sF)$. In particular, $\mu_{\alpha}(\sF) <
  \mu_{\alpha}^{\max}(\sF)$.
\end{assumption}

As an immediate consequence of this assumption, there exists a sequence
$(\sG_j)_{j \in \bN^+}$ of subsheaves $\sG_j \subseteq \sF$ such that the
associated sequence of slopes is increasing and converges,
$$
\lim_{j \to \infty} \mu_{\alpha}(\sG_j) = \mu_{\alpha}^{\max}(\sF) \in \bR.
$$
Given any coherent subsheaf $\sG \subseteq \sF$, its rank will come from the
bounded set $\{0, \ldots, \rank \sF \}$. As a consequence, we may assume that the
ranks of the $\sG_i$ are maximal.

\begin{awlog}[Maximality of rank]\label{awlog:maximality}
  There exists a number $r \in \bN^+$ such that the following two conditions
  hold.
  \begin{enumerate}
  \item For all $i \in \bN^+$, $\rank \sG_j = r$.
  \item Given any number $r'$ and sequence of coherent subsheaves $(\sH_i)_{i
      \in \bN^+}$ such that $\lim_{j \to \infty} \mu_{\alpha}(\sH_j) =
    \mu_{\alpha}^{\max}(\sF)$ and $\rank \sH_i = r'$ for all $i \in \bN^+$, then
    $r' \leq r$.
  \end{enumerate}
\end{awlog}

Replacing the subsheaves $\sG_i \subseteq \sF$ with their saturations, we obtain
a sequence of subsheaves of the same rank $r$, but possibly larger slopes. We
can therefore assume the following.

\begin{awlog}[Saturatedness]\label{awlog:saturated}
  The sheaves $\sG_i$ are saturated subsheaves of $\sF$ for all $i \in
  \bN^+$. In other words, the quotients $\sF/\sG_i$ are torsion-free.
\end{awlog}

None of the assumptions made so far is affected when we replace the sequence
$(\sG_i)_{i \in \bN}$ by subsequence. This allows to assume the following.

\begin{awlog}[Sequence of slopes]\label{awlog:incr}
  The sequence of slopes, $(\mu_{\alpha}(\sG_i))_{i \in \bN^+}$ is strictly
  increasing. Given any number $i \in \bN$, then
  $$
  \mu_{\alpha}(\sG_i) >  \mu_{\alpha}^{\max}(\sF) - \textstyle \frac{1}{i}.
  $$
\end{awlog}

The assumptions made so far imply that none of the sheaves $\sG_i$ is contained
in any one of the $\sG_j$, for $j > i$. We would like to thank the referee for
simplifying the originally somewhat involved argument.

\begin{lem}[Mutual containment]\label{lem:nincl}
  Given any two numbers $i < j$, then the sheaf $\sG_i$ is not contained in
  $\sG_j$.
\end{lem}
\begin{proof}
  Argue by contradiction and assume that there exist numbers $i < j$ and
  inclusions of sheaves $\sG_i \subseteq \sG_j \subseteq \sF$. Dividing by
  $\sG_i$, we obtain an inclusion of quotients
  \begin{equation}\label{eq:incquot}
    \factor \sG_j.\sG_i. \subseteq \factor \sF.\sG_i..
  \end{equation}
  By Assumption~\ref{awlog:maximality}, both $\sG_i$ and its subsheaf $\sG_j$
  have the same rank $r$. The quotient on the left of~\eqref{eq:incquot} is thus
  a torsion sheaf. On the other hand, Assumption~\ref{awlog:saturated} implies
  that the right hand side of~\eqref{eq:incquot} is torsion free. It follows
  that the sheaf on the left must be zero, that is, $\sG_i = \sG_j$. In
  particular, we have an equality of slopes, $\mu_\alpha(\sG_i) =
  \mu_\alpha(\sG_j)$, contradicting Assumption~\ref{awlog:incr}.
\end{proof}

By Assumption~\ref{awlog:saturated}, the sheaves $\sG_i$ are saturated as
subsheaves of $\sF$. Lemma~\ref{lem:nincl} therefore has the following immediate
consequence.

\begin{consequence}\label{cons:rank}
  Given any two numbers $i < j$, let $\sG_i+\sG_j \subseteq \sF$ be the coherent
  subsheaf generated by $\sG_i$ and $\sG_j$. Then $\rank (\sG_i+\sG_j) > \rank
  \sG_i = \rank \sG_j = r$. \qed
\end{consequence}

\subsubsection{Proof of Proposition~\ref*{prop:mumax}, slope computations}

We have seen in Consequence~\ref{cons:rank} that sheaves of the form $\sG_{ij}
:= \sG_i + \sG_j$ have rank larger than $r$. The following lemma shows that
$\mu_\alpha^{\max}(\sF)$ can be approximated by $\sG_{ij}$, for $i$, $j$
sufficiently large. As we will point out in Section~\ref{ssec:popmum5} below,
this violates the Maximality Assumption~\ref{awlog:maximality}, thus finishing
the proof.

\begin{lem}\label{lem:slopecomp}
  Given any two numbers $i < j$, then 
  \begin{equation}\label{eq:slopecomp}
    \mu_{\alpha}(\sG_i + \sG_j) > \mu_{\alpha}^{\max}(\sF) - \left(\textstyle \frac{1}{i} + \frac{1}{j} \right). 
  \end{equation}
\end{lem}
\begin{proof}
  Consider the exact sequence
  $$
  0 \to \sG_i \cap \sG_j \to \sG_i \oplus \sG_j \to \sG_i + \sG_j \to 0. 
  $$
  Since all sheaves involved are torsion-free, hence in particular locally free in
  codimension one, it follows that
  \begin{align}
    \label{eq:addNC} \bigl[ \det(\sG_i + \sG_j) \bigr] & = [\det \sG_i] + [\det \sG_j] - \bigl[\det(\sG_i \cap \sG_j) \bigr], & \text{ and} \\
    \label{eq:addRK} \rank (\sG_i + \sG_j) &= \underbrace{\rank (\sG_i) + \rank (\sG_j)}_{=2r} - \rank (\sG_i \cap \sG_j).
  \end{align}
  In other words,
  \begin{multline*}
    \rank (\sG_i + \sG_j) \cdot \mu_{\alpha}(\sG_i + \sG_j) \\
    \begin{aligned}
      & = r \cdot \mu_{\alpha}(\sG_i) + r \cdot \mu_{\alpha}(\sG_j) - \rank(\sG_i \cap \sG_j) \cdot \mu_{\alpha}(\sG_i \cap \sG_j) && \text{by \eqref{eq:addNC}, \eqref{eq:d1}}\\
      & \geq r \left(\mu_{\alpha}(\sG_i) + \mu_{\alpha}(\sG_j) \right) - \rank(\sG_i \cap \sG_j) \cdot \mu_{\alpha}^{\max}(\sF) && \text{by \eqref{eq:d2}}\\
      & \geq r \left(2\mu_{\alpha}^{\max}(\sF) - \textstyle \frac{1}{i} - \frac{1}{j} \right) - \rank(\sG_i \cap \sG_j) \cdot \mu_{\alpha}^{\max}(\sF) && \text{by Assumption~\ref{awlog:incr}}\\
      & = -r \left( \textstyle \frac{1}{i} + \frac{1}{j} \right) + \rank(\sG_i + \sG_j) \cdot \mu_{\alpha}^{\max}(\sF) && \text{by \eqref{eq:addRK}}
    \end{aligned}
  \end{multline*}
  With this computation in place, Equation~\eqref{eq:slopecomp} follows from
  Consequence~\ref{cons:rank} when we divide by $\rank(\sG_i + \sG_j)$.
\end{proof}

\subsubsection{Proof of Proposition~\ref*{prop:mumax}, end of proof}
\label{ssec:popmum5}

Consider the sequence of coherent sheaves $\sH_j := \sG_j + \sG_{j+1} \subseteq
\sF$. With this definition, Lemma~\ref{lem:slopecomp} asserts that
$$
\lim_{j \to \infty} \mu_{\alpha}(\sH_j) = \mu_{\alpha}^{\max}(\sF).
$$
Passing to a suitable subsequence, we can assume that there exists a number $r'$
such that $\rank \sH_j = r'$, for all $j \in \bN^+$. We have seen in
Consequence~\ref{cons:rank} that $r' > r$, clearly contradicting the Maximality
Assumption~\ref{awlog:maximality}. The contradiction obtained concludes the proof of Proposition~\ref{prop:mumax}. \qed

\subsection{Behaviour of semistability under tensor products}
\label{ssec:AB}

In the setting of Definition~\ref{def:semistability}, if $\sL \in \Pic(X)$ is a
very ample line bundle, if $(H_i)_{1\leq i <\dim X} \in |\sL|$ are general
elements and $\alpha$ is the class of the intersection curve,
$$
\alpha = [H_1 \cap \cdots \cap H_{\dim X-1}],
$$
a classical theorem specific to characteristic zero asserts that the tensor
product of any two semistable locally free sheaves is again semistable,
cf.~\cite[Thm.~3.1.4]{HL97}. This result has been generalized to the case where
$X$ is smooth and $\alpha \in N_1(X)_{\mathbb Q}$ an arbitrary movable class.

\begin{fact}[\protect{Reflexive product preserves semistability on manifolds, \cite[Thm.~5.1 and Cor.~5.5]{CP11}}]\label{fact:SSreflTensor}
  In the setting of Definition~\ref{def:semistability}, assume additionally that
  $X$ is smooth. If $\sF$ and $\sG$ are two torsion free coherent sheaves of
  $\sO_X$-modules, then the following holds
  \begin{enumerate}
  \item\label{il:asx} $\mu_\alpha^{\max}\bigl( (\sF \otimes \sG)^{**} \bigr) =
    \mu_\alpha^{\max}(\sF) + \mu_\alpha^{\max}(\sG)$
  \item\label{il:bsx} If $\sF$ and $\sG$ are $\alpha$-semistable, then $(\sF
    \otimes \sG)^{**}$ is likewise $\alpha$-semistable. \qed
  \end{enumerate}
\end{fact}

We generalize Fact~\ref{fact:SSreflTensor} to the singular case.

\begin{prop}[Reflexive tensor operations preserve semistability on $\mathbb Q$-factorial spaces]\label{prop:SSreflTensor}
  In the setting of Definition~\ref{def:semistability}, let $\sF$ and $\sG$ be
  two coherent sheaves of $\sO_X$-modules that are torsion free in codimension
  one. Then the following holds.
  \begin{enumerate}
  \item\label{il:p161} $\mu_\alpha^{\max}\bigl( (\sF \otimes \sG)^{**} \bigr)
    = \mu_\alpha^{\max}(\sF) +\mu_\alpha^{\max}(\sG)$
  \item\label{il:p162} If $\sF$ and $\sG$ are $\alpha$-semistable, then $(\sF
    \otimes \sG)^{**}$ is $\alpha$-semistable.
  \item\label{il:p163} If $\sF$ is $\alpha$-semistable, then $\Sym^{[q]} \sF :=
    \bigl( \Sym^q \sF \bigr)^{**}$ and $\bigwedge^{[q]} \sF := \bigl(
    \bigwedge^q \sF \bigr)^{**}$ are $\alpha$-semistable, for all $q$.
  \end{enumerate}
\end{prop}

To prove Proposition~\ref{prop:SSreflTensor}, we choose a resolution of
singularities, $\pi : \widetilde X \to X$, and compare semistability of sheaves
on $X$ with semistability of their reflexive pull-back sheaves. The proof of
Proposition~\ref{prop:SSreflTensor}, given on page~\pageref{pf:prop14} below, is
essentially a combination of the invariance lemmas shown in the next section.

\subsection{Invariance properties}
\label{ssec:AC}

Movable curve classes are more flexible than complete intersection curves: given
a singular space $X$ and a resolution of singularities, $\pi : \wtilde X \to X$,
there exists a meaningful notion of pull-back that maps a movable curve class on
$X$ to one on $\wtilde X$. The results of this section discuss stability with
respect to pull-back sheaves and relate stability on $X$ with that on $\wtilde
X$. The following construction is crucial.

\begin{construction}\label{const:pB}
  Let $X$ be a normal, $\mathbb Q$-factorial projective variety and $\pi :
  \wtilde X \to X$ a resolution of singularities. The push-forward map of
  divisors respects $\mathbb Q$-linear and numerical equivalence and therefore
  induces a surjective $\mathbb Q$-linear map
  $$
  \pi_* : N^1(\wtilde X)_{\bQ} \to N^1(X)_{\bQ}.
  $$
  Using the non-degenerate pairing~\eqref{eq:blp}, the dual of $\pi_*$ gives
  rise to an injective map,
  \begin{equation}\label{eq:pb}
    \pi^* : N_1(X)_{\bQ} \to N_1(\wtilde X)_{\bQ},    
  \end{equation}
  which clearly satisfies the projection formula,
  \begin{equation}\label{eq:projection}
    \pi^*(\alpha).\beta = \alpha.\pi_*(\beta), \quad \text{for all
      $\alpha \in N_1(X)_{\bQ}$ and $\beta \in N^1(\wtilde X)_{\bQ}$.}
  \end{equation}
\end{construction}

\begin{rem}
  The map $\pi^*$ of Construction~\ref{const:pB} appears in the literature under
  the name ``numerical pull-back''.
\end{rem}

\begin{lem}
  In the setup of Construction~\ref{const:pB}, if $\alpha \in N_1(X)_{\mathbb
    Q}$ is any class, then $\alpha$ is movable if and only if $\pi^*(\alpha)$ is
  movable.
\end{lem}
\begin{proof}
  The statement follows immediately from the fact that the push-forward and
  strict transform of any effective divisor is always effective.
\end{proof}

We maintain the following setting throughout the remainder of the present
Section~\ref{ssec:AC}.

\begin{setting}\label{set:AC}
  In the setup of Definition~\ref{def:semistability}, let $\pi: \wtilde X \to X$
  be a resolution of singularities. Let $E \subset \wtilde X$ be the
  $\pi$-exceptional divisor, that is, the codimension-one part of the
  $\pi$-exceptional locus. Note that $E$ will be zero if the resolution map is
  small. Finally, set $\wtilde \alpha := \pi^*(\alpha)$, where $\pi^*$ is the
  pull-back map~\eqref{eq:pb}.
\end{setting}

\begin{lem}[Invariance of slope and semi-stability under modifications along the exceptional set]\label{lem:18}
  In Setting~\ref{set:AC}, let $\wtilde \sF$ and $\wtilde \sG$ be two coherent
  sheaves of $\sO_{\wtilde X}$-modules which are torsion free in codimension
  one. Assume that $\wtilde \sF$ and $\wtilde \sG$ are isomorphic outside of the
  $\pi$-exceptional set. Then $\mu_{\wtilde \alpha}(\wtilde \sF) = \mu_{\wtilde
    \alpha}(\wtilde \sG)$.
\end{lem}
\begin{proof}
  Observe that the line bundles $\det \sF$ and $\det \sG$ agree outside of
  $E$. Denoting the irreducible components of $E$ by $E_i$, we can therefore
  write
  $$
  \det \sF \cong \bigl( \det \sG \bigr) \otimes \sO_{\wtilde X} \bigl(
  \textstyle{\sum_i} \lambda_i E_i \bigr), \quad \text{for suitable $\lambda_i
    \in \bZ$}.
  $$
  Since $\pi_*(E_i) = 0$ for all $i$, the equality of slopes then follows from
  the projection formula~\eqref{eq:projection}.
\end{proof}

\begin{lem}[Invariance of slope under push-forward]\label{lem:x7}
  In Setting~\ref{set:AC}, let $\wtilde \sF$ be any coherent sheaf on $\wtilde
  X$ which is torsion free in codimension one.
  \begin{enumerate}
  \item\label{il:lemx71} The push-forward $\pi_* \wtilde \sF$ is again torsion
    free in codimension one.

  \item\label{il:lemx72} We have equality of slopes, $\mu_{\wtilde
      \alpha}(\wtilde \sF) = \mu_{\alpha}\bigl( \pi_* \wtilde \sF \bigr)$.
  \end{enumerate}
\end{lem}
\begin{proof}
  The push-forward of any torsion free sheaf under a surjective map is again
  torsion free.  Claim (\ref{lem:x7}.\ref{il:lemx71}) thus follows from the
  observation that the image of any codimension-two set is again of codimension
  two.

  For the second claim, observe that the sheaves $\pi_* \det \wtilde \sF$ and
  $\det \pi_* \wtilde \sF$ are both torsion free of rank one and agree outside
  the singular set of $X$. Consequently, we have an equality of numerical
  classes on $X$,
  $$
  \pi_* [ \det \wtilde \sF] = [\pi_* \det \wtilde \sF] = [\det \pi_* \wtilde
  \sF] \in N^1(X)_{\bQ}.
  $$
  Equality of (\ref{lem:x7}.\ref{il:lemx72}) thus follows from the projection
  formula~\eqref{eq:projection}.
\end{proof}

\begin{lem}[Invariance of slope under pull-back]\label{lem:17}
  In Setting~\ref{set:AC}, let $\sF$ be any coherent sheaf on $X$ that is
  torsion free in codimension one, and consider its reflexive pull-back
  $\pi^{[*]} \sF := (\pi^* \sF)^{**}$. Then we have equality of slopes,
  $\mu_{\wtilde \alpha}\bigl( \pi^{[*]} \sF \bigr) = \mu_\alpha(\sF)$.
\end{lem}
\begin{proof}
  We know from Claim~(\ref{lem:x7}.\ref{il:lemx72}) of Lemma~\ref{lem:x7} that
  $\mu_{\wtilde \alpha}( \pi^{[*]} \sF) = \mu_{\alpha}( \pi_* \pi^{[*]}
  \sF)$. Since $\sF$ and $\pi_* \pi^{[*]} \sF$ agree in codimension one, we have
  $\det \sF \cong \det \pi_* \pi^{[*]} \sF$. This immediately implies the claim.
\end{proof}

\begin{cor}[Invariance of $\mu^{\max}$ and semistability under pull-back]\label{cor:IV4}
  In Setting~\ref{set:AC}, let $\sF$ and $\wtilde \sF$ be sheaves on $X$ and
  $\wtilde X$, respectively, both torsion free in codimension one. Assume that
  $\wtilde \sF$ is isomorphic to $\pi^*(\sF)$ away from the $\pi$-exceptional
  set.
  \begin{enumerate-p}
  \item\label{il:caspar} Given numbers $(r,\mu) \in \bN \times \bQ$, then $\sF$
    contains a subsheaf $\sG$ of rank $r$ and slope $\mu_\alpha(\sG)=q$ if and
    only if $\wtilde \sF$ contains a subsheaf $\wtilde \sG$ of rank $r$ and
    slope $\mu_{\wtilde \alpha}(\wtilde \sG)=q$.
  \end{enumerate-p}
  In particular,
  \begin{enumerate-p}
    \setcounter{enumi}{\value{equation}}
  \item\label{il:melchior} we have equality $\mu_{\alpha}^{\max} (\sF) =
    \mu_{\wtilde \alpha}^{\max} (\wtilde \sF)$, and
  \item\label{il:balthasar} the sheaf $\sF$ is semistable with respect to $\alpha$
    if and only if $\wtilde \sF$ is semistable with respect to $\wtilde \alpha$.
  \end{enumerate-p}
\end{cor}
\begin{proof}
  Items~(\ref{cor:IV4}.\ref{il:melchior}) and (\ref{cor:IV4}.\ref{il:balthasar})
  are immediate consequences of (\ref{cor:IV4}.\ref{il:caspar}). To prove the
  latter, let $\sG \subseteq \sF$ be any subsheaf with $\rank \sG = r$ and
  $\mu_\alpha(\sG) = q$. Lemma~\ref{lem:17} asserts that $\mu_{\wtilde
    \alpha}(\pi^{[*]}\sG) = q$. By Grothendieck's extension theorem for coherent
  subsheaves, \cite[I.Thm.~9.4.7 and 0.Sect.~5.3.2]{EGA1}, there exists a
  subsheaf $\wtilde \sG \subset \wtilde \sF$ which agrees with $\pi^{[*]}\sG$
  wherever $\pi$ is an isomorphism. It is clear that $\rank \wtilde \sG =
  r$. Lemma~\ref{lem:18} shows that $\mu_{\wtilde \alpha}(\wtilde \sG)=q$.

  Conversely, assume there exists a subsheaf $\wtilde \sG \subseteq \wtilde \sF$
  with $\rank \wtilde \sG = r$ and $\mu_{\wtilde \alpha}(\wtilde \sG) =
  q$. Lemma~\ref{lem:x7} shows that $\mu_{\alpha}(\pi_* \wtilde \sG) = q$, and
  another application of the extension theorem for coherent subsheaves gives the
  existence of a sheaf $\sG \subseteq \sF$ which agrees with $\pi_* \wtilde \sG$
  in codimension one. It is therefore clear that $\det \pi_* \wtilde \sG = \det
  \sG$, so that $\mu_\alpha(\sG) = q$. Since $\rank \sG = r$, this finishes the
  proof of (\ref{cor:IV4}.\ref{il:caspar}) and hence of Corollary~\ref{cor:IV4}.
\end{proof}

\subsection{Proof of Proposition~\ref*{prop:SSreflTensor}}
\label{pf:prop14}

We maintain notation and assumptions of Proposition~\ref{prop:SSreflTensor}. Let
$\pi: \wtilde X \to X$ be a resolution of singularities and set $\wtilde \alpha
:= \pi^*(\alpha)$. Set
$$
\sA := (\sF \otimes \sG)^{**} \quad \text{and} \quad \wtilde \sA := \bigl(
\pi^{[*]}(\sF) \otimes \pi^{[*]}(\sG) \bigr)^{**}.
$$
It follows immediately from the definition that $\wtilde \sA$ is isomorphic to
$\pi^{[*]}(\sA)$ wherever $\pi$ is an isomorphism.

\begin{proof}[Proof of Assertion~(\ref{prop:SSreflTensor}.\ref{il:p161})]
  Immediate from Fact~(\ref{fact:SSreflTensor}.\ref{il:asx}) and
  Corollary~(\ref{cor:IV4}.\ref{il:melchior}).
\end{proof}

\begin{proof}[Proof of Assertion~(\ref{prop:SSreflTensor}.\ref{il:p162})]
  Assume that $\sF$ and $\sG$ are $\alpha$-semistable. In this setup,
  Item~(\ref{cor:IV4}.\ref{il:balthasar}) of Corollary~\ref{cor:IV4} says that
  \begin{itemize}
  \item the sheaves $\pi^{[*]}(\sF)$ and $\pi^{[*]}(\sG)$ are semistable with
    respect to $\wtilde \alpha$, and
  \item $\sA$ is $\alpha$-semistable if and only if $\wtilde \sA$ is $\wtilde
    \alpha$-semistable.
  \end{itemize}
  Semistability of $\wtilde \sA$ being guaranteed by
  Fact~(\ref{fact:SSreflTensor}.\ref{il:bsx}), the claim thus follows.
\end{proof}

\begin{proof}[Proof of Assertion~(\ref{prop:SSreflTensor}.\ref{il:p163})]
  Since the arguments used to prove semistability for symmetric and exterior
  powers are the same, we consider the case of symmetric powers only. Recall
  from \cite[p.~148]{OSS} that there exists a Zariski-open subset $X^\circ
  \subseteq X$ with $\codim_X X\setminus X^\circ \geq 2$, such
  that $\sF|_{X^\circ}$ is locally free. In particular, there exists a direct
  sum decomposition,
  $$
  \sF|_{X^\circ}^{\otimes q} = \Sym^q \sF|_{X^\circ} \,\, \oplus \,\, \factor
  \sF|_{X^\circ}^{\otimes q}.\Sym^q \sF|_{X^\circ}..
  $$
  Using that the complement of $X^\circ$ is small, we obtain a direct sum
  decomposition of reflexive sheaves,
  $$
  \bigl(\sF^{\otimes q}\bigr)^{**} = \Sym^{[q]} \sF \,\, \oplus \,\,
  \left( \factor \sF^{\otimes q}.\Sym^q \sF.
  \right)^{**}.
  $$
  Semistability of $\bigl(\sF^{\otimes q}\bigr)^{**}$ as asserted in
  (\ref{prop:SSreflTensor}.\ref{il:p162}) then shows the last remaining claim,
  finishing the proof of Proposition~\ref{prop:SSreflTensor}.
\end{proof}

\providecommand{\bysame}{\leavevmode\hbox to3em{\hrulefill}\thinspace}
\providecommand{\MR}{\relax\ifhmode\unskip\space\fi MR }
\providecommand{\MRhref}[2]{%
  \href{http://www.ams.org/mathscinet-getitem?mr=#1}{#2}
}
\providecommand{\href}[2]{#2}

\end{document}